\documentclass{amsart}
\usepackage[utf8]{inputenc}
\usepackage[T1]{fontenc}
\usepackage[margin=1.3in]{geometry}  
\usepackage[all,cmtip]{xy}
\usepackage{verbatim,mathtools}
\usepackage{graphicx,wrapfig,enumerate}              
\usepackage{amsmath,amsfonts,amsthm,amssymb}
\usepackage{pifont}
\usepackage{breqn}
\usepackage{tikz} 
\usepackage[most]{tcolorbox}
\usepackage{relsize}
\usepackage{lipsum}
 
\usetikzlibrary{shapes}
\usetikzlibrary{positioning}
\usepackage{ytableau,enumitem}   
\usepackage{xcolor}
\usepackage{float}
\usepackage{pifont}
\usepackage{multirow}

\newtheorem{thm}{Theorem}[section]
\newtheorem{lemma}[thm]{Lemma}

\newtheorem{prop}[thm]{Proposition}

\definecolor{ultragreen}{RGB}{24,120,50} \theoremstyle{remark}
\newtheorem{example}[thm]{\color{ultragreen}Example}
\theoremstyle{plain}

\newtheorem{defn}[thm]{Definition}
\newtheorem{rmk}[thm]{Remark}
\newtheorem{remark}[thm]{Remark}

\newcommand{\target}{\mathrel{\bigcirc\hspace{-2.7mm}\bullet}}

\numberwithin{thm}{section}

\newcommand{\close}{\circlearrowright\! }

\newcommand{\rank}{\mathcal{W}}
\newcommand{\rmm}{\mathcal{M}}
\newcommand{\drm}{\mathcal{M}}
\newcommand{\lloop}{\mathrm{loop}}
\newcommand{\rloop}{\mathrm{loop}_{t}}
\newcommand{\mup}{\mathrm{U}}
\newcommand{\tf}{T_\updownarrow}
\newcommand{\tfm}{\begin{bmatrix}1&-1\\1&0\end{bmatrix}}
\newcommand{\mdo}{\mathrm{D}}
\newcommand{\mupm}[1]{\begin{bmatrix}#1 &1\\0&1 \end{bmatrix}}
\newcommand{\mdom}[1]{\begin{bmatrix} 1+ #1 & -#1\\1&0 \end{bmatrix}}
\newcommand{\definition}[1]{{\color{coral}\emph{#1}}}
\newcommand{\ds}{\displaystyle}

\definecolor{coral}{RGB}{245, 93, 159}
\definecolor{magenta}{RGB}{164, 93, 245}
\definecolor{teal}{RGB}{28, 157, 186}
\definecolor{pea}{RGB}{118, 186, 28}
\definecolor{violet}{RGB}{138,43,226}

\usepackage[colorinlistoftodos]{todonotes}

\title[ Cluster Expansions]{Cluster Expansions: T-walks, Labeled Posets and Matrix Calculations}
\author{Ezgi Kantarcı Oğuz and Emine Yıldırım}

\begin{document}
\maketitle

\begin{abstract}
   We give two new combinatorial methods for computing cluster expansion formulas for arcs coming from possibly punctured surfaces. The first is by using \emph{T-walks}, an extension of $T$-path models from~\cite{S10,ST09} for unpunctured surfaces to general surfaces. To do so, we introduce a new combinatorial way to generate these paths. The second is by using order ideals of labeled posets associated to arcs. In this context, we use the methods introduced in~\cite{O22,OR21} to give a quick way to calculate the expressions using $2$ by $2$ matrices. The techniques introduced are applicable to different settings in cluster algebras and beyond. 
\end{abstract}

\tableofcontents

\section{Introduction}

Cluster algebras were introduced by Fomin-Zelevinsky~\cite{FZ02} in early 2000 in the study of Lusztig's dual canonical bases and total positivity in semisimple groups. They are commutative rings with a distinguished set of generators called \definition{cluster variables}. Usually rings are given by set of generators and relations, however cluster algebras are given by a set of \definition{initial cluster variables} and an iterative rule, called \definition{mutation}, to obtain the rest of generators. Cluster algebras can be defined in many different ways such as using matrices, directed graphs which we call quivers, surfaces and are widely studied in different areas~\cite{BMRRT06, CCS06, FG06, FG09, FST08, FST12, FZ02, FZ07, GLFS22}. We will work with the surface definition following~\cite{FST08} which initiated the study of cluster algebras arising from triangulations of a surface. A triangulation gives a set of initial cluster variables and mutation can be seen as going from a triangulation to another by flipping arcs. Thus, cluster variables will be arcs in the surface we work with. Since the invention of cluster algebras, people are interested in trying to give direct and easier formulae to find the expansion of cluster variables from the initial ones~\cite{BMR09, BK20, BZ11, CK06, CK08, CL12, CS13, CS15, CT19, GLFS22, GM15, L09, MP07, M11, P08, Ra18, Z07}. This is especially important to understand Lusztig's dual canonical bases elements in terms of cluster algebra elements. Thanks to the fact that the cluster variables corresponding to the arcs can be interpreted as certain "\definition{lambda length}" in the decorated Teichmüller space~\cite{FG06, FG09, FT18}, one can express the expansions of cluster variables as a product of matrices in $PSL_2(\mathbb{R})$. In~\cite{MW13}, Musiker-Williams associate products of matrices $PSL_2(\mathbb{R})$ to the \emph{generalized} arcs and closed loops. Our approach in this paper again uses matrices from $PSL_2(\mathbb{R})$ but we have a different approach to compute the expansions.

The idea of calculating the expansion formulae via matrices is not a new, in fact it might be thought of as predating the combinatorial methods such as looking at matchings in snake graphs or ideals of posets. It has not been the go-to method for doing the calculations however, possibly because the framework of snake graphs proved more accessible. In the meantime, with the definition of the new $q$-deformations of rational numbers by Morier-Geoaud and Ovsienko \cite{MG020}, the combinatorics of fence posets and corresponding polynomials came into a new attention of mathematical community, prompting new works on the combinatorical aspects and calculation methods (\cite{McCSS21},\cite{O22},\cite{OR21}). This work aims to bring these developments back into the cluster setting and give a matrix characterization that can be fully visualized as building posets step by step. The biggest advantage of this method is that the ideas are easy to extend to new settings in cluster algebras and beyond. We will, in addition to the classical cases of snake and band graph, give a characterization of calculating expansion formulae for Wilson's loop graphs \cite{wilson} using matrices. We believe that the methods we use can be adapted to the framework of these papers~\cite{MOZ21, MOZ22} and plan to pursue this in the future.

Schiffler and Thomas \cite{ST09} gave expansion formulae of cluster variables using certain paths, (complete) $T$-paths, on a triangulation of an unpunctured surface. Schiffler in \cite{S10} generalized this to the cluster algebras with coefficients again in the unpunctured setting. Later on, Gunawan and Musiker \cite{GM15} studied these path for a surface only with one puncture. There are further works in the literature about $T$-paths, see for instance \cite{CJ20,BCKZ22}. We consider a generalization of $T$-paths associated to generalized arcs on all surfaces possibly with punctures in this paper. Moreover, our construction includes the computation of coefficients in the cluster algebra we work with. As already in literature, there is a bijection between perfect matchings and $T$-paths for the unpunctured surface is shown by Musiker-Schiffler \cite{MS10}, we furthermore show that this bijection holds in general. 


We start by recalling some well-known constructions in Section~\ref{sec:prelim}. Then in Section~\ref{sec:Twalks} we give our first expansion formula using \emph{$T$-walks}:

{\bf Theorem~\ref{thm:texpansion}} The cluster expansion formula for an arc $\gamma$ (possibly ordinary, closed, singly or doubly notched) on a punctured surface can be calculated via T-walks as follows:
\begin{align*}x_{\gamma}=\displaystyle  \sum_{T_{\vec{v}} \in \mathrm{TW}(\gamma)} x(T_{\vec{v}})y(T_{\vec{v}}).
\end{align*} 

$T$-walks are quite similar in flavor to $T$-paths; for $T$-walks some repetitions are allowed to accommodate the notched cases. In fact, by removing these repetitions in a $T$-walk, we obtain the corresponding unique $T$-path (Proposition~\ref{T-bijection}). 

In Section~\ref{sec:labeled_posets} we give another formulation; this time in terms of labeled posets. Calculating cluster expansions in this manner is a natural and straightforward approach. In essence, we directly extract the combinatorial information required for the expansion from the surface, bypassing the need to transform it into snake graphs as an intermediary step.

{\bf Theorem.~\ref{thm:poset extension}} Let $P_\gamma$ be the labeled poset associated to an arc $\gamma$ (possibly ordinary, closed, singly or doubly notched). Then the expansion of $x_{\gamma}$ is given by:
\begin{align*}x_{\gamma}=\displaystyle \frac{x(M_{-})}{cross(T,\gamma)} \, \, \rank(P_\gamma;xy)
\end{align*}

In Subsection~\ref{subsec:proof}, we explore how two different methods of computing cluster expansions are related. Specifically, even though our two models may seem quite distinct, there is a very natural bijection between them (See Theorem~\ref{thm:bijection}).

We want to mention that following the initial sharing of this work on the arXiv, there has been increased interest in utilizing posets for computing cluster expansions. In particular, in a recent paper~\cite{PRS23} by Pilaud, Reading and Schroll, they write cluster variables using labeled posets in a complete general setting. They provide proofs which uses hyperbolic structure of the surface. Moreover, labeled posets appear in~\cite{W23} to study certain Donaldson-Thomas $F$-polynomials.

The rest of this work is devoted to giving matrix methods to efficiently compute the cluster expansion via matrices.  For this, we make use of  \emph{oriented posets} a combinatorial tool recently defined by Kantarcı~Oğuz. In Section ~\ref{sec:matrix}, after recalling the necessary notation from \cite{O22}, we give efficient matrix formulas for the different types of arcs we are considering. To illustrate our matrices in action, we include full calculations for different types of arcs in the appendix. 

We further want to note that, though all our methods work for the degenerate case where the initial triangulation includes some self-folded triangles, there are some technicalities involved that we chose to address in a separate section. See Section~\ref{sec:selffold} for  details on how the expansion formulas work in the self-folded case.

\subsection*{Acknowledgement} EKO greatfully acknowledges support by T\"{U}B\.{I}TAK 2218 Grant \#121C385 and by T\"{U}B\.{I}TAK 1001 Grant \#123F121. EY is supported in part by the Royal Society Wolfson Award RSWF/R1/180004.

\section{Preliminaries}\label{sec:prelim}

\subsection{Surface Cluster Algebras}

Let $S$ be an orientable, connected, compact $2$-dimensional real surface  with (possibly empty) boundary $\partial S$. Let us assume positive orientation of $S$ as counterclockwise.

We fix a nonempty set $M$ of marked points on $S$ or $\partial S$ such that each boundary component of $S$ contains at least one point from $M$. Marked points not on $\partial S$ are called \definition{punctures}. The pair $(S,M)$ is called surface with marked points. 

We exclude the cases when the pair $(S,M)$ is a sphere with one, two or three punctures; a monogon with zero or one puncture; and a bigon or triangle without punctures as in~\cite{FT18}.

An ordinary arc (simply called \definition{arc}) $\gamma \in (S,M)$ is a simple curve on $S$ (i.e. does not cross except endpoints), considered up to isotopy, between two marked points in $M$  and $\gamma$ that does not cut out a monogon or a digon. Note also that an arc does not intersect the boundary except possibly at its endpoints. We say that two arcs in $S$ are \definition{noncrossing} if the arcs can be drawn with no intersection except possibly at their endpoints.

An \definition{ideal triangulation} $T$ of $(S,M)$ is a maximal collection of pairwise noncrossing arcs inside $S$ where no arc is isotopic to a boundary segment. The triangulation $S$ with the boundary segments, partitions $S$ into triangles. However, the two sides of a triangle may not necessarily be distinct, as illustrated in Figure~\ref{fig:self-fold} below. Such a triangle is called \definition{a self-folded triangle}. Self folded triangles are formed by a monogon encircling a puncture point and a \definition{radius} arc connecting the puncture to the marked point on the monogon.

\begin{figure}[ht]
\includegraphics[width=7.7cm]{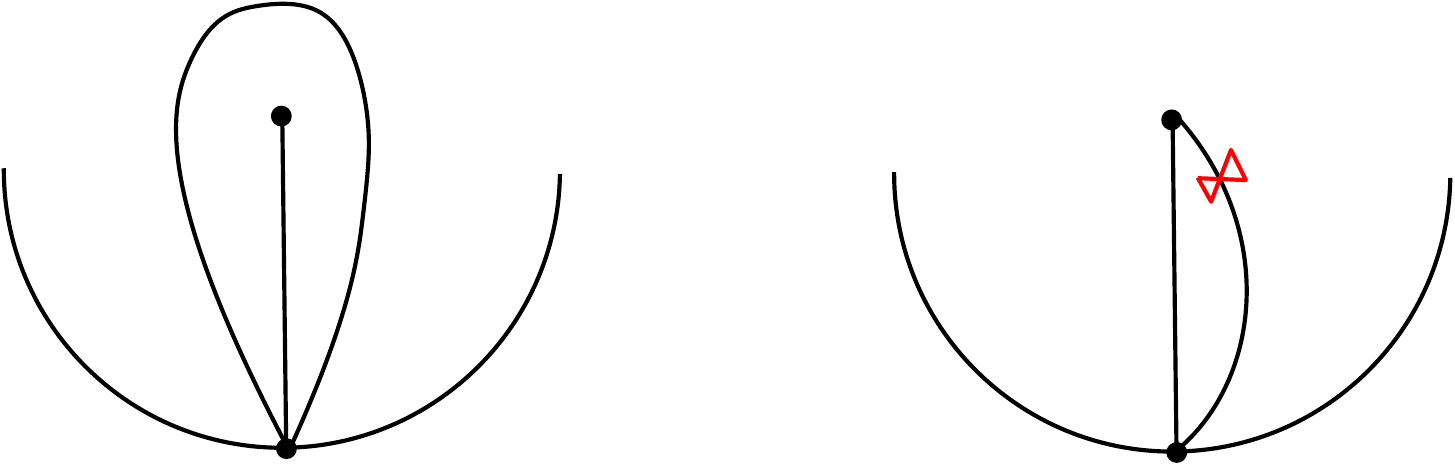}
\caption{An illustration of a self-folded triangle on the left and a tagged arc on the right.} \label{fig:self-fold}
\end{figure}

To be able to flip the arcs inside self-folded triangles, we need a new definition: a \emph{tagged} arc.

\begin{defn}
A \definition{tagged arc} is an arc that does not cut out an once-punctured monogon, with each of its ends marked either \emph{plain} or \emph{tagged}, such that every endpoint on $\partial S$ marked plain and if the start and end of an arc coincide then they are both marked the same way. 
\end{defn}

The definition of being noncrossing can be extended to tagged arcs.

\begin{defn} Two tagged arcs$\gamma$ and $\gamma'$ are said to be \definition{noncrossing} if one of the following is satisfied:
    \begin{enumerate}
    \item Untagged versions of $\gamma$ and $\gamma'$ are not isotopic, noncrossing and any shared endpoints are marked the same way in both tagged arcs.
    \item Untagged versions of $\gamma$ and $\gamma'$ are isotopic and exactly \emph{one} endpoint of $\gamma$ is marked in the same way as the corresponding endpoint of $\gamma'$.
\end{enumerate}
\end{defn}

A maximal pairwise noncrossing collection of tagged arcs is called a \definition{tagged triangulation}. Throughout this work, any use of the word triangulation will refer to a tagged triangulation, unless stated otherwise.

In this section, we will give some results on \emph{cluster algebras} which can be found in the papers suggested below. The aim of this paper is to give combinatorial formulae for finding expansions of cluster algebra elements. For the sake of combinatorics in this paper, it is enough to work with the geometric model. Thus we will not need a detailed definition of cluster algebras. We refer the reader to ~\cite{FZ02, FZ07, FT18, FST08} for the definitions and extraordinarily beautiful properties of cluster algebras and references in there. We also refer to \cite{MSW11, MW13, wilson} for the background on expansion formulae in cluster algebras coming from triangulated surfaces. 

For a cluster algebra, we have a set of generators, called \definition{cluster variables}, which are grouped into certain subsets, called \definition{clusters} or \definition{seeds}. Mutations in this setting are just flipping arcs in a triangulation to get another triangulation of the marked surface. The ground ring for cluster algebras is the group ring $\mathbb{ZP}$ of a semifield (sometimes called tropical semifield) $\mathbb{P}.$ If $\mathbb{P}=1$, we says that the cluster algebra has trivial coefficients. Otherwise, the cluster algebra have so called principle coefficients. Briefly, a cluster algebra with principle coefficients is the $\mathbb{ZP}$-subalgebra of a field of rational functions $\mathcal{F}:=\mathbb{QP}(x_1,\ldots,x_n)$ and comes with the following data; $(\boldsymbol{x},\boldsymbol{y},B)$ where
\begin{itemize}
    \item $\boldsymbol{x}=(x_1,\ldots,x_n)$ is a tuple of algebraically independent variables of $\mathcal{F}$, called \definition{initial cluster variables};
    \item $\boldsymbol{y}=(y_1,\ldots,y_n)$ is a tuple of generators of $\mathbb{P}$, called \definition{initial coefficients};
    \item a so-called skew-symmetric $B$-matrix or a quiver $Q$ without loops and $2$-cycles. (Here we note that Fomin-Zelevinsky~\cite{FZ02} defined cluster algebras in a more general setting using skew-symmetrizable matrices.)
\end{itemize}

\begin{thm}[\cite{FT18}] Let $(S,M)$ be a surface with marked points and we set an initial triangulation of the surface $S.$ Then there is a cluster algebra $\mathcal{A}$ associated to this surface which has the following properties:
\begin{itemize}
    \item the seeds are in bijection with the tagged triangulations of $(S,M)$;
    \item the cluster variables are in bijection with tagged arcs in $(S,M)$; and 
    \item the cluster variable $x_{\gamma}$ associated to an arc $\gamma$ is given by the lambda length of the arc $\gamma.$
\end{itemize}
    
\end{thm}

\subsection{Snake Graphs, Loop Graphs, Band Graphs}

Given a surface $S$ with a triangulation $T$ and an arc $\gamma$ on $S$, one would like to compute the expansion formula for the corresponding cluster variable in an efficient way. For this purpose, many authors suggested various methods. We will review the ones that first generate a combinatorial object from the given data (a snake graph, a band graph or a loop graph) and then express the expansion formula in terms of perfect matchings on that graph.

We will review the connection between these graphs and their corresponding expansion formulae, as well as how these graphs and their perfect matchings are in bijection with certain posets and their order ideals. Our goal in this paper, after generating the posets from snake graphs or directly from surfaces, to express the expansion formulae using the combinatorics of the posets and calculating the expansion formulae as products of $2$ by $2$ matrices.

Snake graphs appear in \cite{P05} as combinatorial objects assigned to triangulations of polygonal surfaces, also in \cite{CS13, CS15, MSW11} in the context of cluster algebras.  Band graphs appear in \cite{MSW13} as a combinatorial tool to assign a cluster expansion to closed loops on a surface $(S,M)$. Loop graphs appear in Definition 3.7 and 4.7 of \cite{wilson} in an attempt to parameterize expansion formula for tagged arcs. Band or loop graphs are constructed from a snake graph by gluing some edges. One considers good matchings on the resulting graphs which are matchings that can be extended to a perfect matching of the corresponding snake graph.   

For the detailed definitions and constructions, see \cite{MSW11, MSW13, wilson}.

\begin{defn}
\begin{enumerate}
    \item A perfect matching $M$ of a snake graph $G$ uses a fixed subset $\{i_1,i_2,\ldots,i_k\}$ of edges in $G.$ Then the \definition{weight} of $M$, $x(M)$, is defined as $x_{i_1}x_{i_2}\ldots x_{i_k}$ which is obtained by multiplying the labels on the edges of $M$.
    \item A snake graph has exactly two perfect matching that only include boundary edges of $G$; we called them \definition{minimal} and \definition{maximal} perfect matchings, $M_{-}$ and $M_{+}$ for a fixed convention. We follow the standard convention for $M_{-}$ which is starting with the south edge of the first tile for a perfect matching.
    \item Let $M$ be a perfect matching of a snake graph $G$. The set $\mathcal{M}:=(M_{-}\cup M)\setminus (M_{-}\cap M)$ is the collection of some boundary edges which is a subgraph $G_M$ (possibly disconnected) of $G$. Then we can define the \definition{coefficient monomial}, $y(M)$ of $M$ as the product of all $y_i$ where the tile $G_i$ lies inside $\mathcal{M}$ possibly with multiplicities.
\end{enumerate}
    
\end{defn}

Here is a simple example. 
\begin{figure}[H]
    \centering
    \includegraphics[width=7.9cm]{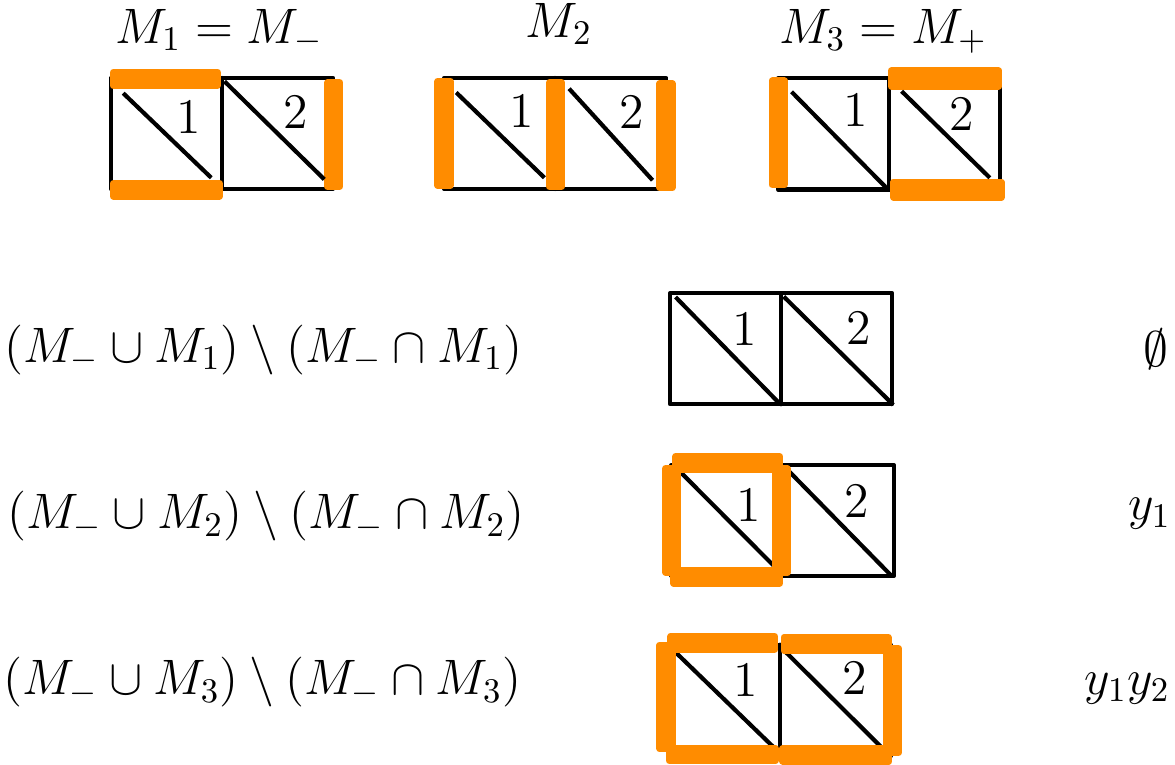}
    \caption{}
    \label{}
\end{figure}

\begin{rmk}\label{rmk:1} We need to define coefficient monomial, $y(M)$ of $M$, more carefully in the case where we have self-folded triangles in the initial triangulation of the surface. For an arc inside the self-folded triangle, say $r$, we consider the $y$ variable as $y_r/y_l$ where $l $ is the loop arc around the arc $r$. We refer reader to \cite{MW13} for details. See Section~\ref{sec:selffold} for more discussion of the self-folded case.
\end{rmk}

\begin{rmk}\label{rmk:2}The notion of perfect matching of a snake graph can be extended to band and loop graphs, in that case they are called \definition{good} matchings by \cite{wilson}. Briefly, a good matching on a snake graph is just a matching that can be extended to a perfect matching of the corresponding band or loop graph. \end{rmk}

\begin{thm}~\cite{FZ07, MSW11, wilson}\label{thm:snake} The expansion formula for an arc $\gamma$ has a general form

\begin{center}
    $\ds x_\gamma=\frac{1}{cross(T,\gamma)}\sum_M x(M)y(M) $
\end{center}
where $M$ ranges over a certain set of matchings of the associated snake (band or loop) graph $G$.

Here
\begin{itemize}
    \item $\ds cross(T,\gamma)=\prod_{j=1}^{k}x_{i_j}$ is the monomial obtained by multiplying the labels of the edges of the triangulation $T$ crossed by $\gamma$,
    
    \item $\ds x(M)=\prod_{j=1}^{r}x_{i_j}$  is the weight of the perfect (good) matching $M$ of the associated graph (snake, band or loop),
    
    \item $\ds y(M)=\prod_{j\in J}y_{i_j}$ is the coefficient monomial of the perfect (good) matching of the underlying graph (snake, band or loop).
\end{itemize}
\end{thm}

\section{Expansions via T-paths and T-walks}~\label{sec:Twalks}

In this section we will extend the definition of T-paths previously given for ordinary arcs and loops  to all surfaces, including punctured surfaces notched arcs. To this end, we will make use of an intermediary structure that we call \emph{T-walks}, which are similar to T-paths but they allow revisiting arcs. Let us start with defining T-paths in the ordinary case.

Let $(S,M)$ be a surface and $T$ be a triangulation. Recall that we work with Riemann surfaces whose positive orientation assumed to be counterclockwise. Set $\gamma$ be an ordinary arc between marked points on S which is not in $T$. We fix the orientation of $\gamma$ such that it has a start and an end point. Assume $\gamma$ crosses some arcs in $T$ $n$-times.

\begin{defn} A $(T,\gamma)$-path $\alpha=(\alpha_1,\cdots,\alpha_n)$, $T$-path for short, is a concatenation of arcs $\alpha_1,\cdots,\alpha_n$ of the triangulation, where the start and end of a $T$-path align with $\gamma$; and $\alpha$

\begin{enumerate}
    \item has to cross $\gamma$ at least once,
    \item can use both boundary and internal arcs of the surface,
    \item is of odd length,
    \item cannot reuse an arc as it follows $\gamma$,
    \item at every even step of the concatenation it must cross $\gamma$.
\end{enumerate} 
\end{defn}

Note that it may cross $\gamma$ at the odd steps. Also, the places where it crosses $\gamma$ are in order along $\gamma$. We associate a Laurent polynomial $x(\alpha)$ to a $(T,\gamma)$-path $\alpha$, which is a cluster algebra element, as follows.

\[x(\alpha)=\prod_{i \text{ odd}} x_{\alpha_i}\prod_{i \text{ even}} x^{-1}_{\alpha_i}\]

\begin{thm}~\cite{ST09} Let $\mathcal{T}$ denote the set of all $(T,\gamma)$-path for a given arc $\gamma\in T$. Then
    \[x_{\gamma}=\sum_{\alpha\in\mathcal{T}}x(\alpha)\] is the cluster variable correnponding to an arc $\gamma$ in an unpunctured surface $(S,M).$
\end{thm}

Let $\gamma$ be an arc on a (possibly punctured) surface, crossing the arcs in $T$ in order $a_1,a_2,\ldots,a_n$ and label triangles along the way with $\triangle_i$ starting from $i=0$. Here we allow $\gamma$ to be a notched arc or a loop for full generality. Therefore, $a_i$'s may not be all distinct. Let us assign a direction to each arc $a_i$ crossed by $\gamma$ by orienting $a_i$ towards one of its endpoints. Since $\gamma$ crosses $n$ arcs, there are $2^n$ possible choices of assigning the directions to the collection of arcs crossed by $\gamma$. We will encode choices using vectors  $\vec{v} \in \{0,1\}^n$, where $v_i=0$ if $a_i$ is oriented counterclockwise as the orientation of   $\triangle_{i-1}$, and $v_i=1$ otherwise. (See Section~\ref{sec:selffold} for the self- folded case). We call the direction $\vec{0}$ the \definition{minimal direction}. See Figure~\ref{fig:vyon} below the examples, the example on the left having arc $3$ crossed twice with each crossing oriented separately.

\begin{figure}[H]
\includegraphics[width=9.3cm]{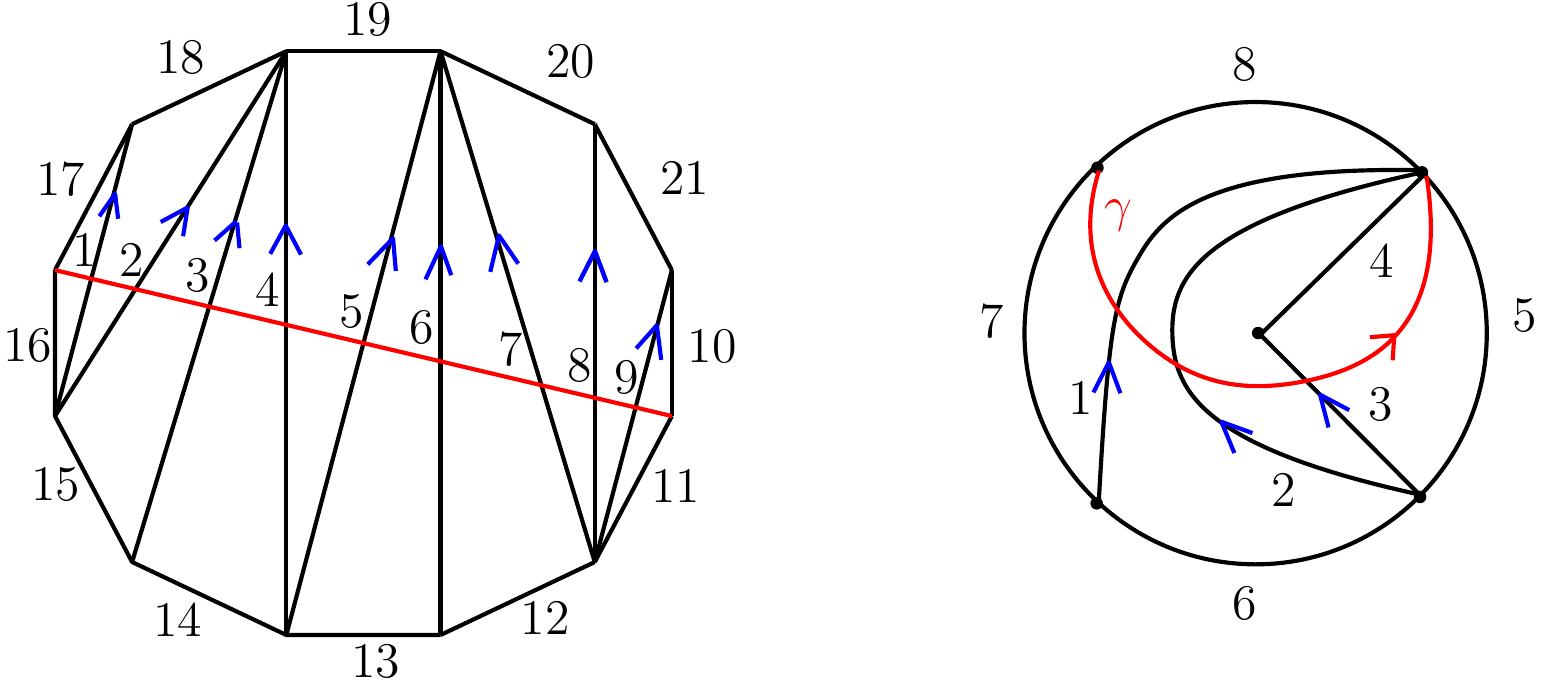}
\caption{The minimal direction $\vec{v}$ for chosen $\gamma$} \label{fig:vyon}
\end{figure}

\begin{lemma}{\label{lem:minimal direction}} In the minimal direction $\vec{0}$, if the arc $a_{i+1}$ follows $a_i$ counterclockwise, then $a_i$ and $a_{i+1}$ are both directed towards the shared vertex in the $\triangle_i$. They are both directed away from the shared vertex in the $\triangle_i$ otherwise.
\end{lemma}

\begin{figure}[H]
\includegraphics[width=5.6cm]{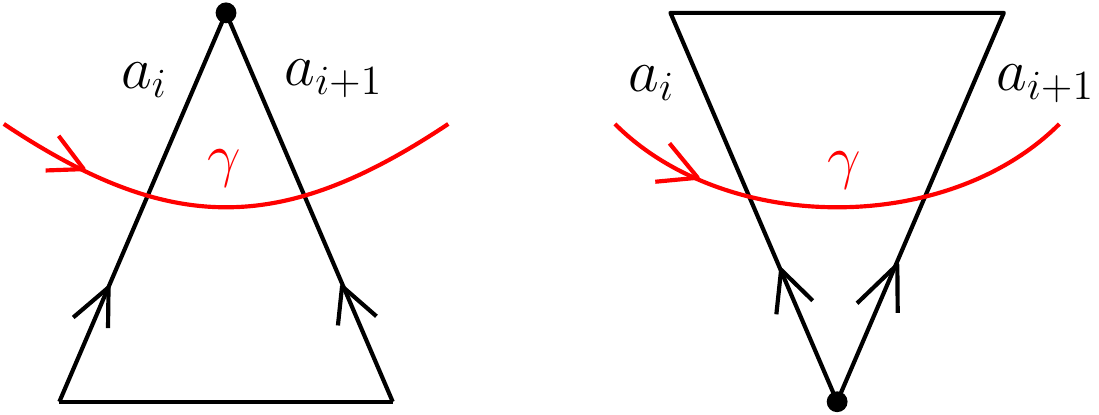}
\caption{}
\end{figure}

We now give a \emph{T-walk} definition which generalize the definition of a T-path. Earlier definitions can be found in~\cite{S10} for unpunctured surfaces and \cite{GM15} for once-punctured disk. Informally, a T-walk is a certain walk which is formed by concatenating arcs on a given triangulated surface respecting the order of crossings of the arc $\gamma$ as in the case of a T-path but now we allow to use an arc more than one. Note that we drop the notation $\gamma$, but when we say a T-walk we mean a T-walk for a given arc $\gamma$. 

\begin{defn} For an arc $\gamma$ with crossings $e_1,e_2,\ldots, e_n$ and a direction $\vec{v} \in \{0,1\}^n$, let $\alpha_i$ denote the arc with crossing $e_i$, directed according to $\vec{v}$. If we can find arcs $\beta_1,\ldots,\beta_{n-1}$ (adding $\beta_0$ and $\beta_n$ if necessary) satisfying the following properties, we call the concatenation of arcs  ($\beta_0$) $\alpha_1,\beta_1,\alpha_2,\beta_2,\cdots,\beta_{n-1},\alpha_n$ ($\beta_n$) the \definition{T-walk} $T_{\vec{v}}$. Else, we say there is no T-walk corresponding to $\vec{v}$:
\begin{itemize}
    \item For $i\in [n-1]$, $\beta_i$ is the unique directed arc in $\triangle_i$  that connects the end point of $\alpha_i$ with the starting point of $\alpha_{i+1}$. 
    \item If $\gamma$ is a loop, $\beta_n$ is the unique directed arc in $\triangle_0$ connecting the $\alpha_n$ to the starting point of $\alpha_1$. There is no $\beta_0$ in this case to consider.
    \item If $\gamma$ is not a loop, $\beta_0$ exists if and only if the starting point of $\gamma$ and the starting point of $\alpha_1$ are distinct. It is the unique directed arc in $\triangle_0$ that connects the starting point of $\gamma$ to the starting point of $\alpha_1$.
    \item Similarly, if $\gamma$ is not a loop, $\beta_n$ exists if and only if the endpoint of $\gamma$ and the end point of $\alpha_b$ are distinct. In that case, $\beta_n$ is the unique directed arc in $\triangle_n$ that connects the end point of $\alpha_n$ to the end point of $\gamma$.
\end{itemize} 
\end{defn}

\begin{rmk} Note that each choice of $\vec{v}$ gives you at most one valid T-walk, as in any triangle $\triangle_i$ there is a single arc connecting any distinct points (See Section \ref{sec:selffold} for self-folded case). Furthermore, the T-walk $T_{\vec{v}}$, if it exists, has the following properties:
    \begin{enumerate}
        \item  The T-walk $T_{\vec{v}}$ begins at the starting point of $\gamma$ and ends at the end of $\gamma$, which may be the same point if $\gamma$ is a loop.
        \item One has to cross $a_i$ at every other step of the concatenation.
        \item An arc in $T$ can be used several times in a walk; including using an arc back and forth repetitively.
    \end{enumerate}
\end{rmk}

We use $\mathrm{TW}(\gamma)$ to denote the set of all T-walks corresponding to $\gamma$. As in the case of T-paths, we associate an element of the cluster algebra to each T-walk $T_{\vec{v}}$ as follows:
$$x(T_{\vec{v}}):=\displaystyle\frac{\prod x_{\beta_i}}{\prod x_{\alpha_i}}.$$
We further set $y_{\vec{v}}:=\prod_{{v_i}=1} y_{\alpha_i}$. 

Note that the denominator $\prod x_{\alpha_i}$ is the same for all $T$-walks and is equal to $cross(T,\gamma)$.

\begin{example}~\label{ex:Twalk} The following example can be compared with Example 3.9 in~\cite{GM15}. Consider the triangulation of the disk with an arc $\gamma$ as shown in Figure~\ref{fig:twalks}.

\begin{figure}[H]
    \centering
    \includegraphics[width=3.5cm]{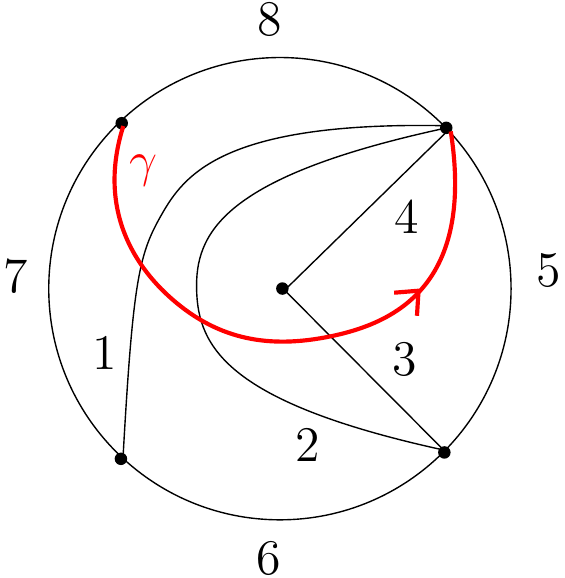}
    \caption{An arc $\gamma$ crossing three arcs.}
    \label{fig:twalks}
\end{figure}

Since $\gamma$ crosses three arcs, there are $2^3=8$ possible vectors $\vec{v}$ for the orientation of arcs $a_1, a_2$ and $a_3$ which are crossed by $\gamma$. The five valid $T$-walks are drawn below. Notation is simplied by just using numbers $i$ instead of $a_i$'s.

\begin{figure}[H]
    \centering
    \includegraphics[width=15cm]{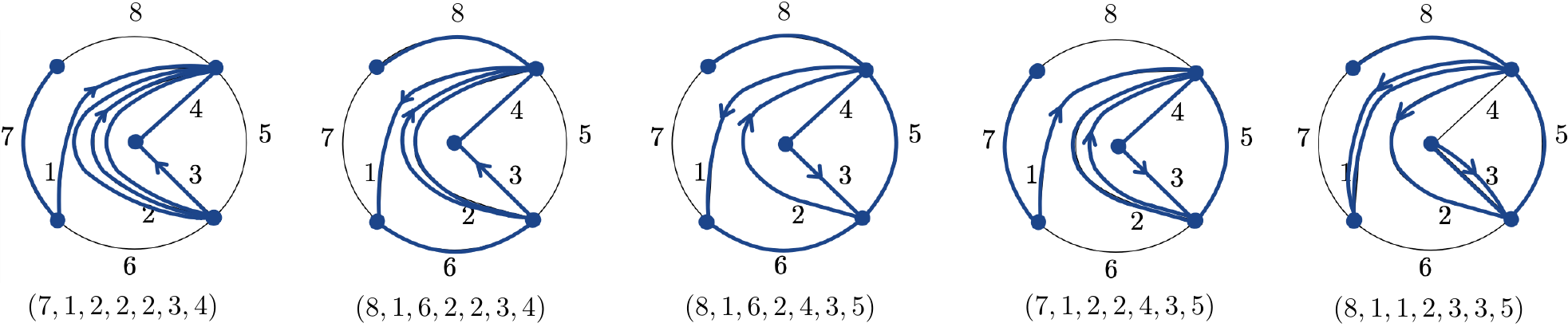}
\end{figure}

 Note that choices $(1,1,0),(0,1,1)$ and $(0,1,0)$ do not give valid T-walks. That is because the if the endpoint of $\alpha_1=1$ and the starting point of $\alpha_2=2$ are the same as in $(0,1,1)$, there is no valid choice for $\beta_1$ on $\triangle_1$. Similarly the endpoint of $\alpha_2=2$ and the starting point of $\alpha_3$ being the same creates a problem, as in the case of $(1,1,0)$. The case $(0,1,0)$ has both these issues.   In general, though there are $2^n$ available directions, the cardinality of the set of T-walks will be a lot smaller, matching the number of terms in the cluster expansion.  
 
For each of vectors that actually give valid $T$-walks, the corresponding walks, $x_{\vec{v}}$ and $y_{\vec{v}}$ values are listed below. 
\begin{itemize}
    \item $\vec{v}_1=(0,0,0)$: $\quad$ $T_{\vec{v}_1}=(7,1,2,2,2,3,4)$,$\quad$  $x_{\vec{v}_1}=\frac{x_2 x_4 x_7}{x_1 x_3}$,$\quad$  $y_{\vec{v}_1}=1$.
    \item $\vec{v}_2=(1,0,0)$: $\quad$ $T_{\vec{v}_2}=(8,1,6,2,2,3,4)$,$\quad$ $x_{\vec{v}_2}=\frac{x_4 x_6 x_8}{x_1 x_3}$,$\quad$  $y_{\vec{v}_2}=y_1$.
    \item $\vec{v}_3=(0,0,1)$: $\quad$ $T_{\vec{v}_4}=(7,1,2,2,4,3,5)$,$\quad$ $x_{\vec{v}_2}=\frac{x_4 x_5 x_7}{x_1 x_3}$,$\quad$  $y_{\vec{v}_2}=y_3$.
    \item $\vec{v}_4=(1,0,1)$: $\quad$ $T_{\vec{v}_4}=(8,1,6,2,4,3,5)$,$\quad$  $x_{\vec{v}_4}=\frac{x_4 x_6 x_8}{x_1 x_3}$,$\quad$  $y_{\vec{v}_4}=y_1y_3$.
    \item $\vec{v}_5=(1,1,1)$: $\quad$  $T_{\vec{v}_5}=(8,1,1,2,3,3,5)$,$\quad$  $x_{\vec{v}_5}=\frac{x_5 x_8}{x_2}$,$\quad\hspace{.32cm}$  $y_{\vec{v}_4}=y_1 y_2 y_3$.

\end{itemize}

\end{example}

We will see next how the values above can be used to calculate the cluster expansion of our arc.

 \begin{thm}\label{thm:texpansion} The expansion formula for a generalized arc $\gamma$ on a possibly punctured surface can be calculated via T-walks as follows:
\begin{align}x_{\gamma}=\displaystyle  \sum_{T_{\vec{v}} \in \mathrm{TW}(\gamma)} x(T_{\vec{v}})y(T_{\vec{v}}).
\end{align}\label{eqn:main2} 
\end{thm}

The proof of the Theorem is postponed to Section \ref{subsec:proof}.

\begin{example} \label{ex:twalksviatwalks} For the surface from Figure~\ref{fig:twalks}, the 5 possible $T$-walks were listed in Example~\ref{ex:Twalk} above. We can use this to calculate $x_{\gamma}$:

\begin{align*}
x_{\gamma}&=\displaystyle  \frac{x_2 x_4 x_7}{x_1 x_3}+\frac{x_4 x_6 x_8 y_1}{x_1 x_3} +\frac{x_4 x_5 x_7 y_3}{x_1 x_3} +\frac{x_4 x_6 x_8 y_1 y_3}{x_1 x_3} +\frac{x_5 x_8 y_1 y_2 y_3}{x_2}
\\ &=\displaystyle \frac{1}{x_1 x_2 x_3} \left({x_2^2 x_4 x_7}+{x_2 x_4 x_6 x_8 y_1}+{x_2 x_4 x_5 x_7 y_3} +{x_2 x_4 x_6 x_8 y_1 y_3} +{ x_1 x_3 x_5 x_8 y_1 y_2 y_3}
\right)
\end{align*}
\end{example}


Earlier in the literature, Schiffler~\cite[Theoem 3.1]{S10} proved such a result for arcs in the unpunctured surfaces.
  
 Furthermore, $T_{\vec{v}}$ crosses $\gamma$ at each $\alpha_i$, which are either all the even steps or all the odd steps. It is also possible to have T-walks of even length. This is a necessary generalization to accommodate different cases, such as notches. The biggest difference between T-walks and T-paths is that, in T-walks arcs are often reused, but that can easily be remedied.

 For a T-walk $T_{\vec{v}}$ cancelling pairs of redundancies (steps where we follow a vertex and go back immediately) eventually gives us a simpler walk which we denote by $\lfloor T_{\vec{v}} \rfloor$. Note that the order of cancellations does not matter. Next we show that in settings where T-paths are defined, what we end up with is a T-path 

\begin{prop}~\label{T-bijection} Let $\gamma$ be an ordinary arc. For any T-walk $T_{\vec{v}}$ corresponding to $\gamma$, $\lfloor T_{\vec{v}} \rfloor$ is a T-path. In fact, the operation  $\lfloor \cdot \rfloor$ defined above gives a bijection between T-walks and T-paths. Furthermore we have:
\begin{equation*}
    x(T_\epsilon)=x(\lfloor T_\epsilon \rfloor).
\end{equation*}
\end{prop}
\begin{proof} First note that as $\gamma$ does not contain any notches (and arcs are considered up to isotopy) the end points of $\gamma$ are not on the arcs $\alpha_1$ and $\alpha_n$. That means all T-walks corresponding to $\gamma$ are of form $ T_\epsilon=(\beta_0,\alpha_1,\beta_1,\ldots,\alpha_n,\beta_n$).

As cancellations always involve consecutive pairs, $\lfloor  T_\epsilon \rfloor$ is of odd length and every even step is given by an $\alpha_i$ entry correponding to a crossing of $\gamma$. 

Furthermore, as we are considering only the minimal number crossings $\gamma$ up to isotopy, it does not revisit any triangle of the triangulation. That means do not reuse any $\alpha_i$ and any $\beta_i$ can only match $\alpha_{i-1}$ or $\alpha_{i+1}$ in which case they get cancelled out. Note that $\beta_0$ and $\beta_n$ can not be cancelled. That means no arc is reused, and $\lfloor  W(\epsilon) \rfloor$ crosses $\gamma$ at least once. 

Pairwise cancellations of matching entries do not affect the $x(T_\epsilon)$ value, $x(T_\epsilon)$ is given by $x$ variables corresponding to even indexed entries of $\lfloor  T_\epsilon \rfloor$ divided by $x$ variables corresponding to odd indexed entries of $\lfloor  W(\epsilon) \rfloor$.

    The fact that this correspondence is a bijection can be recovered from Theorem~\ref{thm:poset extension} and Theorem~\ref{thm:texpansion}, where we show both models are in bijection with poset ideals.

\end{proof}

\begin{example} The corresponding $T$-paths to the $T$-walks listed in Example~\ref{ex:Twalk} are as follows.

\begin{itemize}
    \item $(7,1,3,4)$,
    \item $(8,1,6,3,4)$,
    \item $(8,1,6,2,4,3,5)$,
    \item $(7,1,4,3,5)$,
    \item $(8,2,5)$.
\end{itemize}
\end{example}

\begin{figure}[H]
    \centering
    \includegraphics[width=16cm]{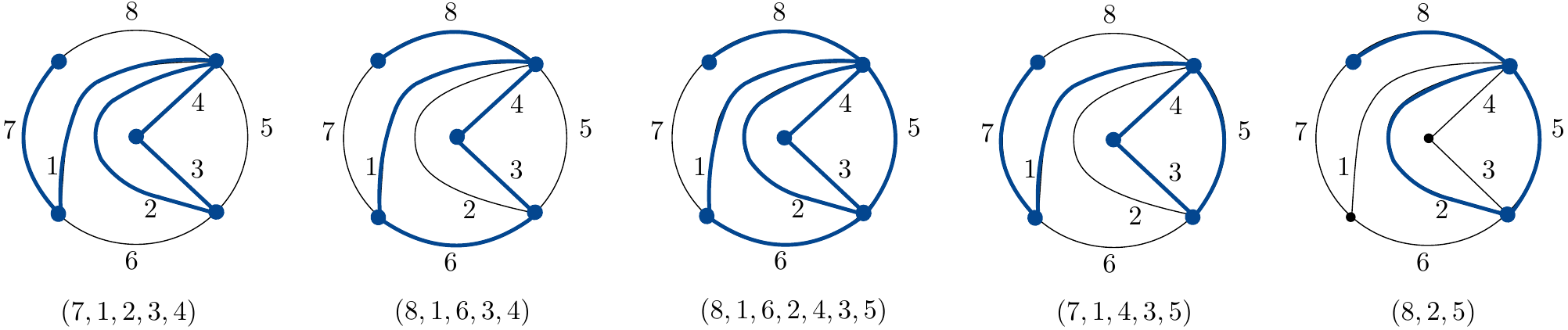}
    \caption{T-paths drawn on the disk.}
    \label{fig:twalks3}
\end{figure}

\section{Expansions via labeled posets}~\label{sec:labeled_posets}

In this chapter, we describe an alternative way to calculate the expansion of an arc by defining a corresponding labeled poset. Once this is achieved, the actual calculations can be done easily by matrix multiplication, using the theory of oriented posets.

Let $P$ denote a partially ordered set; we call $P$ a \emph{poset} for short. An \emph{order ideal} $I$ in $P$ is a down-closed subset of $P$, i.e. for every $x\in I$, $y$ is also in $I$ if and only if $y\leq x$ for $y\in P.$ 

Let $P$ be a poset and $J(P)$ be the order ideal lattice of $P.$  We will associate each node in $P$ by a variable $w_i$ called the \definition{weight} of $i$, and for each $I\in J(P)$, we will let the \definition{weight} of $I$, denoted $w(I)$, be given by the product of the weight of the vertices in $I$. The \definition{weight polynomial} of $P$ is defined by:
  	\begin{align*}
		\rank(P;w)=& \sum_{I\in J(P)} \prod_{i\in I}w_i= \sum_{I\in J(P)} w(I).
	\end{align*}

  Of particular interest to us are a family of posets called \emph{fence posets}. For a composition $\alpha=(\alpha_1,\alpha_2,\ldots,\alpha_k)$ of $n$, the corresponding \definition{fence poset } is a poset on $n+1$ elements, whose Hasse diagram is given by $\alpha_1$ up steps, followed by $\alpha_2$ down steps, followed by $\alpha_3$ up steps...
	\begin{figure}[ht]
	\begin{tabular}{c c}
	\begin{tikzpicture}[scale=.8]
\draw (0,0)--(1.5,1)--(6,-2);
\fill (0,0) circle(.1) node[above,yshift=.17cm] {$1$} ;
\fill (1.5,1) circle(.1) node[above,yshift=.17cm] {$2$} ;
\fill (3,0) circle(.1) node[above,yshift=.17cm] {$3$} ;
\fill (4.5,-1) circle(.1) node[above,yshift=.17cm] {$4$} ;
\fill (6,-2) circle(.1) node[above,yshift=.17cm] {$5$} ;
\end{tikzpicture}& \raisebox{15mm}{\begin{tabular}{c} The poset on the left has the  ideals:\\
\\ $\varnothing$, $\{1\}$, $\{5\}$, $\{1,5\}$,\\ $\{4,5\}$, $\{1,4,5\}$, $\{3,4,5\}$,\\ $\{1,3,4,5\}$, $\{1,2,3,4,5\}$. 
\end{tabular}}
	\end{tabular}
	    \centering

\caption{The fence poset for $\alpha=(1,3)$ (left) and its ideals (right).}\label{fig:ex1}
\end{figure}
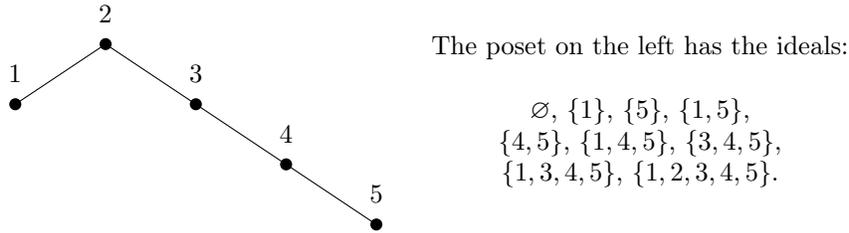

\begin{example}\label{ex:1} The poset given in Figure~\ref{fig:ex1} has nine ideals in total, listed next to it. Accordingly, one gets the weight polynomial:
\begin{dmath*}
1+w_1+w_5+w_1w_5+w_4w_5+w_1w_4w_5+w_3w_4w_5+w_1w_3w_4w_5+w_1w_2w_3w_4w_5.
\end{dmath*}
\end{example}


Consider a triangulation $T$ and an arc $\gamma$ which is not in the triangulation. Label the crossings of $\gamma$ by $a_1,a_2,\ldots,a_m$ in order, where at crossing $a_i$, $\gamma$ crosses an arc $e_i$ from a triangle $\triangle_{i-1}$ to a triangle $\triangle_i.$ Note that the same arc may be crossed multiple times and the triangles visited are not necessarily distinct). 

We put a poset structure $P_{\gamma}$ on the set of crossings $\{a_1,a_2,\ldots,a_m\}$ by taking the transitive closure of the following relations:

\begin{itemize}
    \item For each consecutive pair $a_i$ to $a_{i+1}$, add the relation $a_i\succeq a_{i+1}$ if the arc $e_{i+1}$ follows the arc $e_{i}$ in the counterclockwise orientation in $\triangle_i$, $a_i \preceq a_{i+1}$ otherwise. 
(See Section \ref{sec:selffold} for how this is resolved in case of self-folded triangles.) 
\item If the arc is a loop, also add the relation $a_1\succeq a_{n}$ if the arc $e_{1}$ follows the arc $e_{n}$ in the counterclockwise orientation in $\triangle_n$, $a_n \preceq a_{1}$ otherwise. 
\item If the crossings $a_1,a_2,\ldots,a_{t_1-1}$ (resp. $a_{t_2+1},a_{t_2+2},\ldots,a_m$)  are formed by the resolution of a notch in the sense of Wilson \cite{wilson}, add a relation for the pair $a_1,a_{t_1}$ (resp. $a_{t_2}$, $a_k$) using the same procedure.
\end{itemize}

Our definition extends the definition given in \cite{ST09,S10} to include loops and notched graphs. 

Now we assign a labeling to the vertices of $P_\gamma$: For each vertex $a_i$, we define a monomial $x(a_i)$ and $y(a_i)$ as follows. The monomial $x(a_i)$ is given by

$$\displaystyle x(a_i):=\frac{x_{\alpha} x_{\beta}}{x_{\sigma} x_{\epsilon}},$$

where $\alpha$ and $\beta$ are the labels of the arcs that follow the arc $e_i$ counterclockwise and $\sigma$ and $\epsilon$ are the arcs that follow $e_i$ clockwise in the triangles $\triangle_i$ and $\triangle_{i-1}$. (Any contribution from arcs in a self-folded triangle can be taken to be $1$ because of cancellations.)

We set $y(a_i)=y_{e_i}$ unless $e_i$ is a radius of a self folded triangle. If $e_i$ is a radius, we define $y(a_i)=y_\ell/y_{e_i}$ where $\ell$ is the label of the loop around $a_i$. See Section~\ref{sec:selffold} for more on the subject.



Now we are ready to state the second main result of this work.

\begin{thm}~\label{thm:poset extension} Let $(S,M)$ be a as usual and $\gamma$ an be arc on a triangulation $T$ of $(S,M)$. Let $P_\gamma$ be the labeled poset associated to $\gamma$. Then the expansion of $x_{\gamma}$ with respect to the triangulation $T$ is given by:
\begin{align}x_{\gamma}=\displaystyle {x(T_0)} \, \, \rank(P_\gamma;xy)
\end{align}\label{eqn:main}
where $T_0$ is the minimal $T$-walk of $\gamma$.
\end{thm}

\begin{remark}\label{remark:clusterversion} For the readers more familiar with the language of snake graphs, Equation~\ref{eqn:main} can be reformulated as 
\begin{align*} 
x_{\gamma}= \displaystyle \frac{x(M_{-})}{cross(T,\gamma)} 
 \rank(P_{\gamma};xy).
\end{align*} 
\end{remark}

To give a straightforward proof of Theorem~\ref{thm:poset extension}, we will make use of the theory of snake graphs, which we will recall in the next subsection, where the notation from the above remark will also be made clear. We end this subsection with the poset calculations for the example from Figure~\ref{fig:twalks}.

\begin{figure}[H]
    \centering
    \includegraphics[width=3.5cm]{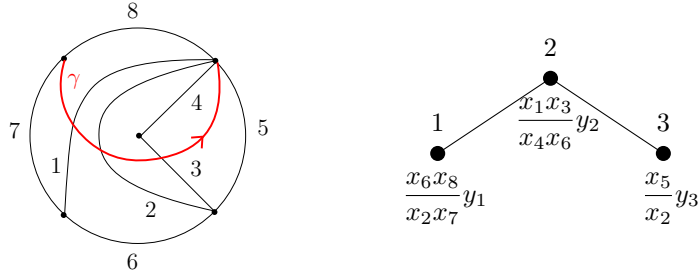} \qquad \qquad\raisebox{.5cm}{ \begin{tikzpicture} 
\fill (3,0) circle(.1)node[above,yshift=.17cm] {$1$} ;
\draw (3.1,-.6) node{$\displaystyle \frac{x_{6} x_{8}}{x_2 x_{7}}y_1$};
\fill (4.5,1) circle(.1) node[above,yshift=.17cm] {$2$} ;
\draw (4.6,.4) node{$\displaystyle \frac{x_{1}x_{3}}{x_{4}x_{6}}y_2$} ;
\fill (6,0) circle(.1) node[above,yshift=.17cm] {$3$} ;
\draw (6.1,-.6) node{$\displaystyle \frac{x_{5}}{x_{2}}y_3$} ;
\draw (3,0)--(4.5,1)--(6,0);
    \end{tikzpicture}}
    \caption{The arc $\gamma$ from Figure~\ref{fig:twalks} with its corresponding labeled poset.}
    \label{fig:twalks4}
\end{figure}
\begin{example}  \label{ex:twalksviaposets} The poset given on the right in Figure~\ref{fig:twalks4} has $5$ ideals: $\varnothing$, $\{1\}$, $\{3\}$, $\{1,3\}$, $\{1,2,3\}$. It has the weight polynomial
  	\begin{align*}
		\rank(P_{\gamma};w)= 1 +w_1+w_3+w_1w_3+w_1w_2w_3.
	\end{align*}
 Plugging in the labels gives us:
   	\begin{align*}
		\rank(P_{\gamma};xy)= 1 +{\displaystyle \frac{x_{6} x_{8}}{x_2 x_{7}}y_1}+\frac{x_{5}}{x_{2}}y_3+ \frac{x_5 x_{6} x_{8}}{x_2^2 x_{7}}y_1 y_3+ \frac{x_1 x_3 x_5 x_{8}}{x_2^2 x_4 x_{7}}y_1 y_2 y_3.
	\end{align*}
Multiplying this with our previously calculated value $x(T_0)=(x_2x_4x_7)/(x_1x_3)$ from Example~\ref{ex:twalksviatwalks} gives:
  	\begin{align*}
		\displaystyle x_\gamma &= \frac{x_2x_4x_7}{x_1x_3} +{\frac{x_4 x_{6} x_{8}}{x_1 x_{3}}y_1}+\frac{x_4 x_5 x_{7}}{x_{1}x_3}y_3+ \frac{x_4 x_5 x_{6} x_{8}}{x_1 x_2 x_{3}}y_1 y_3+ \frac{x_5 x_{8}}{x_2}y_1 y_2 y_3\\
  \displaystyle &= \frac{1}{x_1 x_2 x_3} \left( x_2^2x_4x_7 +x_2x_4x_6x_8y_1+x_2x_4x_5x_7y_3+x_4x_5x_6x_8y_1y_3+x_1x_3x_5x_8y_1y_2y_3\right).
	\end{align*}
 Note that this matches our calculations from Example~\ref{ex:twalksviatwalks}.
\end{example}













This theorem is written an over-simplified form here, as discussed in Remarks~\ref{rmk:1} and \ref{rmk:2} above For self-folded triangles in the initial triangulation, coefficient monomial $y(M)$ is expressed differently, see Section~\ref{sec:selffold} below. The interested reader can refer to \cite[Theorem 5.4]{MSW11} for snake graphs, \cite[Theorem 5.7]{MSW11} for band graphs and \cite[Theorem 7.9]{wilson} for loop graphs. 

\subsection{Connection to order ideals of posets}

The perfect matchings of a snake (band or loop) graph $G$ are in bijection with the order ideals of a poset $P$, where snake graphs correspond to fence posets. The correspondence between a snake graph and a fence poset is as follows: Diagonals inside the tiles of a snake graph are the vertices of the poset and edges are between consecutive tiles: The direction of an edge is obtained as we go along the tiles by changing direction only if the tiles are align in the same direction in the snake graph.

We note that as we can read fence posets from snake graph, we can also read the corresponding fence posets from triangulations of the surface by putting one vertex for each diagonal in the triangulation and drawing a directed edge between vertices with respect to the orientation of the surface if they are the sides of the same triangle. On a related note, we can say that order ideals in fence posets can be thought as the submodules of a certain module over the quivers of the fence posets. We refer reader to \cite{CS21} to find out more about perfect matchings of a snake graph versus submodules of quivers of fence posets.

Any gluing on a snake graph to obtain band or loop graph of $G$ will introduce new directed edges between the corresponding vertices in the fence poset. See Examples~\ref{fig:band},~\ref{singlenotched},~\ref{doublynotched},~\ref{doubly2}. We note that posets corresponding to band or loop graphs can be easily obtained from the triangulated surface as well. 

The set of perfect (good) matchings of a snake (band or loop) graph $G$ has a poset structure, where two perfect matchings are connected if one can be obtained from the other by rotating two markings in a tile. This poset is isomorphic to the poset of order ideals (with respect to inclusion) in the corresponding fence poset. One can say that counting perfect matchings of a graph translates into counting order ideals of a fence poset. 

\begin{figure}[H]
\includegraphics[width=11cm]{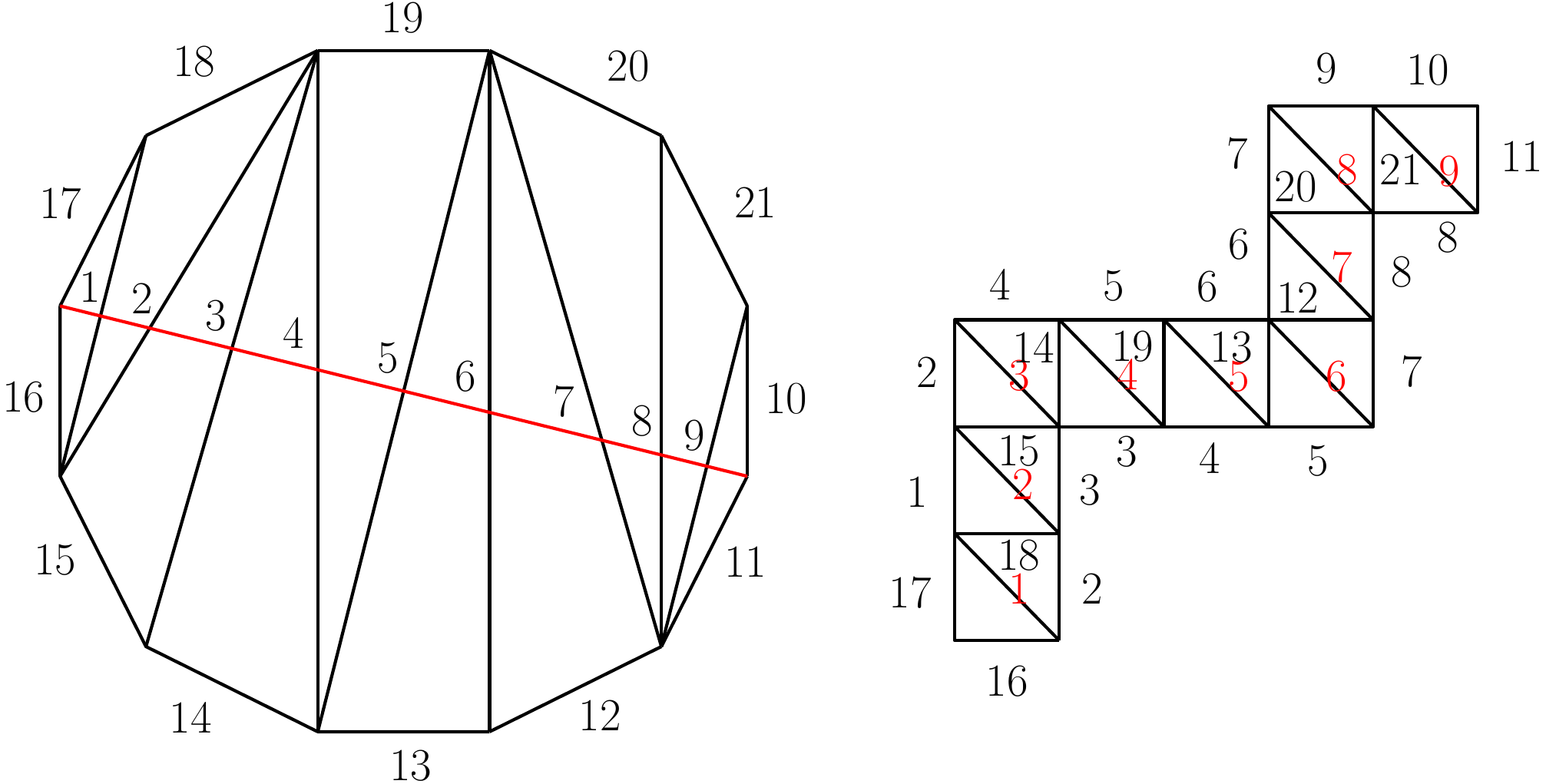}
\caption{An example of an arc $\gamma$ in an initial triangulation and its snake graph} \label{fig:sch10}
\end{figure}

\begin{figure}[H]
\includegraphics[width=9cm]{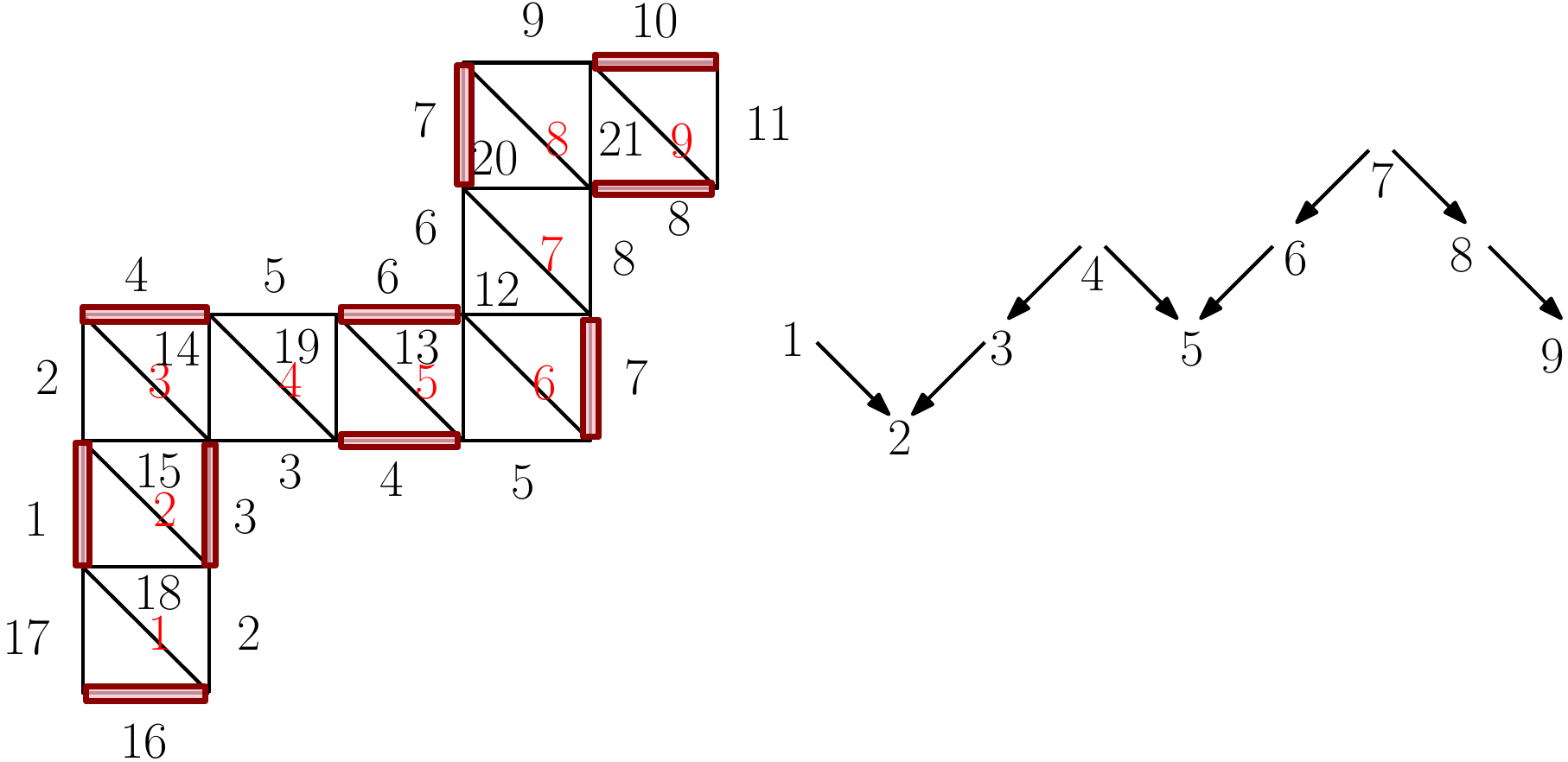}
\caption{The minimal perfect matching and the fence poset associated to the arc $\gamma$} \label{fig:min}
\end{figure}

\subsection{Proof of Theorem~\ref{thm:poset extension}}

Consider a poset coming from a snake graph $G$ with vertices in $[k]$. Let $N(i)$, $S(i)$, $E(i)$, $W(i)$ and $D(i)$ denote the labels of the north, south, east, west and diagonal edges of the tile $G_i$ of $G$ that is matched to vertex $a_i$ of the poset. Note that in snake graphs we have the additional structure of signs, where boxes with sign $-$ are flipped. The $x$ and $y$ variables we have previously defined for the vertices can be described as follows in terms of the snake graph labels:
\begin{align}
x(a_i)&= \begin{cases}\displaystyle \frac{x_{E(i)}x_{W(i)}}{x_{N(i)}x_{S(i)}} &\text{if box $G_i$ has sign $+$,}\\ \\
\displaystyle \frac{x_{N(i)}x_{S(i)}}{x_{E(i)}x_{W(i)}} &\text{if box $G_i$ has sign $-$}
\end{cases} \label{eqn:x}\\
y(a_i)&:= \begin{cases}
 y_{r}y^{-1}_{l} &\text{if $D(i)=r$ is a radius and $l$ is the loop around it}\\ 
 y_{D(i)}& \text{otherwise.} 
\end{cases} \label{eqn:y}
\end{align}

\begin{prop} \label{prop:minimalexists} Consider a curve $\gamma$. The $T$-walk $T_0$ for the minimal direction always exists. Furthermore, if $M_-$ is the minimal matching of the snake graph for $\gamma$, then $x(T_0)=\frac{x(M_-)}{cross(T,\gamma)}$ 
\end{prop}
\begin{proof}  By definition, $M_-$ can be obtained via directing the diagonals upwards if the orientation $+$, downwards if the orientation is $-$, and then connecting them.  Note that the diagonals in the snake graph model are the crossings $a_i$, and this as downwards squares are rotated, so this is equivalent to directing each crossing $i$ crosswise in $\triangle_{i-1}$ that we do for the minimal $T$ paths. The $\beta$ edges of the $T$-path that connect these crossing edges are exactly the edges in the snake path connecting the directed diagonals (See Figure~\ref{fig:minimalscompared} for an example). Thus, such edges exist and $\prod_i x_{\beta_i}=x(M_-)$. We divide both sides by $cross(T,\gamma)=\prod_i x_{\alpha_i}$ to get our result.
\end{proof}

\begin{figure}[H]
\includegraphics[width=11cm]{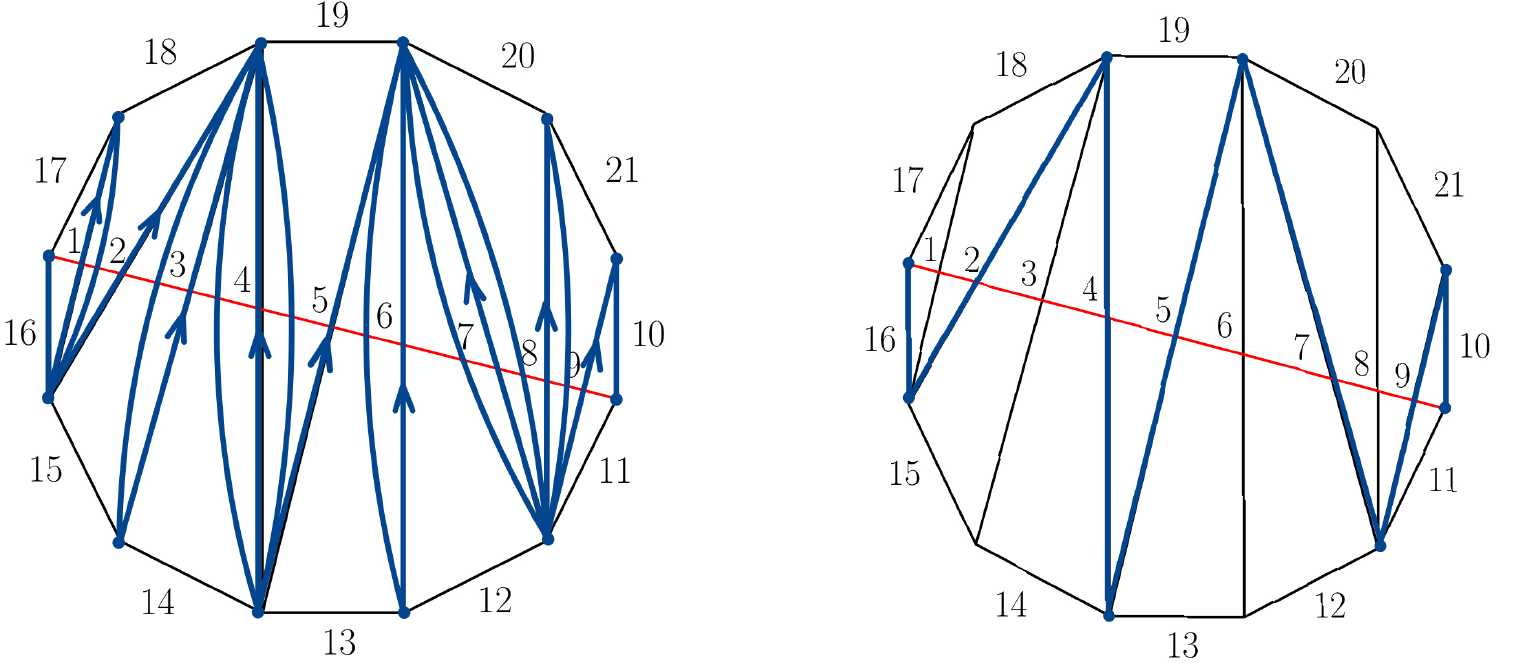}
\caption{The minimal $T$-walk $(16,1,1,2,3,3,4,4,4,5,6,6,7,7,7,8,8,9,10)$ and the minimal $T$-path for the minimal matching in Figure~\ref{fig:min}.} \label{fig:minimalscompared}
\end{figure}

Consider the evaluation of the weight polynomial at $w_i=x(i)y(i)$. We can use it to calculate the cluster expansion formula corresponding to an arc $\gamma$.

\begin{proof}[Proof of Theorem~\ref{thm:poset extension}] 
Let $M$ be a perfect (good) matching of the snake graph $G$, corresponding to the ideal $h(M)$ of the corresponding (fence) poset $P$. The idea here is that the lattices of the perfect matchings of $G$ and the order ideals in $P$ are the same~\cite{CS21}. Each edge in order ideal lattice is given by the label of each tile as in Equation~\ref{eqn:x}. We can obtain $M$ from the minimal matching $M_-$ by a series of moves associated to $i\in h(M)$. Each move $i$ corresponds to multiplying the $x(M)$ value by the rational function $x(i)$ defined in Equation~\ref{eqn:x} above. Consequently we have:
\begin{align*}
x(M)=x(M_-)\prod_{i\in h(M)}x(i).
\end{align*}
Similarly, we obtain coefficient variables by labeling each edge in the order ideal lattice by the y variable of the corresponding tile as we move up in the lattice. Thus, we obtain $y(M)=\prod_{i\in h(M)}y(i)$. Summing over all perfect (good) matchings $M$ or equivalently over all order ideals we get:
\begin{align*}
\sum_{M}x(M)y(M)=x(M_-)\sum_{I}\prod_{i\in I}x(i)y(i)=\rank(P_\gamma;xy).
\end{align*} Result follows by Theorem~\ref{thm:snake} and Proposition~\ref{prop:minimalexists}. \end{proof}

\begin{example}\label{ex:noloop2} Consider the example illustrated in figures \ref{fig:sch10}, \ref{fig:min} and \ref{fig:minimalscompared} above. The minimal matching gives us  $x(M_-)=x_1 x_3 x_4^2 x_6 x_7^2 x_8 x_{10} x_{16}$. One can also draw the corresponding fence poset, with each node $i$ labeled by its weight $x(i)y(i)$ as follows. 

 \begin{center}\scalebox{.8}{
\begin{tikzpicture}
\draw (0,0)--(1.5,-1)--(4.5,1)--(6,0)--(9,2)--(12,0); 
\fill (0,0) circle(.1) node[above,yshift=.17cm] {$1$} ;
\draw (.1,-.6) node{$\displaystyle \frac{x_2 x_{17}}{x_{16}x_{18}} y_1$} ;
\fill (1.5,-1) circle(.1) node[above,yshift=.17cm] {$2$} ;
\draw (1.6,-1.6) node{$\displaystyle \frac{x_{15} x_{18}}{x_1 x_{3}}y_2$};
\fill (3,0) circle(.1) node[above,yshift=.17cm] {$3$} ;
\draw (3.2,-.6) node{$\displaystyle \frac{x_2x_{14}}{x_4x_{15}}y_3$} ;
\fill (4.5,1) circle(.1) node[above,yshift=.17cm] {$4$} ;
\draw (4.6,.4) node{$\displaystyle \frac{x_{3}x_{5}}{x_{14}x_{19}}y_4$} ;
\fill (6,0) circle(.1) node[above,yshift=.17cm] {$5$} ;
\draw (6.1,-.6) node{$\displaystyle \frac{x_{13}x_{19}}{x_{4}x_{6}}y_5$} ;
\fill (7.5,1) circle(.1) node[above,yshift=.17cm]{$6$} ;
\draw (7.7,.4) node{$\displaystyle \frac{x_{5}x_{12}}{x_7x_{13}}y_6$} ;
\fill (9,2) circle(.1) node[above,yshift=.17cm] {$7$} ;
\draw (9.1,1.4) node{$\displaystyle \frac{x_{6}x_{8}}{x_{12}x_{20}}y_7$} ;
\fill (10.5,1) circle(.1) node[above,yshift=.17cm] {$8$} ;
\draw (10.6,.4) node{$\displaystyle \frac{x_{9}x_{20}}{x_{7}x_{21}}y_8$} ;
\fill (12,0) circle(.1) node[above,yshift=.17cm] {$9$} ;
\draw (12.1,-.6) node{$\displaystyle \frac{x_{11}x_{21}}{x_{8}x_{10}}y_9$};
\end{tikzpicture}}
\end{center}

 For the expansion formula we get:
 
  \begin{align*}
  x_{\gamma}=\displaystyle \frac{x_1x_3x_4^2x_6x_7^2x_8x_{10}x_{16}}{x_1x_2x_3x_4x_5x_6x_7x_8x_9} \rank(P_\gamma;xy)=\displaystyle \frac{x_4x_7x_{10}x_{16}}{x_2x_5x_9} \rank(P_\gamma;xy).
\end{align*}

Here, the  $\rank(P_\gamma;xy)$ consists of $64$ terms corresponding to the $64$ perfect matchings of the snake graph and a calculation by hand is impractical. We will discuss a faster way of doing these calculations using matrices in Section~\ref{sec:matrix}.
 \end{example}

\subsection{Bijection to T-walks}\label{subsec:proof}

 We now give a bijection between $T$-walks of $\gamma$ and ideals of the poset $P_\gamma$ that we will use to prove Theorem~\ref{thm:texpansion}.

\begin{thm}\label{thm:bijection} The map $\Gamma:\mathrm{TW}(\gamma)\rightarrow J(P_\gamma)$ mapping each T-walk $T_{\vec{v}}$ to the ideal of $P_\gamma$ given by $\{x_i\mid {\vec{v}}_i=1\}$ is a well-defined bijection. Furhtermore we have:
\begin{equation}
    x(T_{\vec{v}})=x(T_0)\cdot x(\Gamma(T_{\vec{v}})), \qquad \qquad \mathrm y(T_{\vec{v}})=y(\Gamma(T_{\vec{v}})).
\end{equation}
\end{thm}

\begin{proof} Consider the extension map $\overline{\Gamma}:\{0,1\}^n\rightarrow 2^{P_\gamma}$ mapping any ${\vec{v}}$ to the subset $\{x_i\mid {\vec{v}}_i=1\}$ of $P_\gamma$. We will show that the image of an ${\vec{v}}$ under this map is an ideal of $P_\gamma$ if and only if a T-walk $T_{\vec{v}}$ exists.
Note that, by the definition of $P_\gamma$, $i\succeq i+1$ if and only if $e_{i+1}$ follows $e_i$ counterclockwise in $\triangle_i$. In that case, there is an arc in $\triangle_i$ connecting the endpoint of $e_i$ to the starting point of $e_{i+1}$ unless they are both directed counterclockwise. 
As the $0$ direction is given by directing $e_{i+1}$ counterclockwise in $\triangle_i$ and $e_i$ counterclockwise in $\triangle_{i-1}$, this happens if and only ${\vec{v}}_{i+1}=0$ and ${\vec{v}}_i=1$. Similarly, if $i\preceq i+1$, there is no arc in $\triangle_i$ connecting the endpoint of $e_i$ to the starting point of $e_{i+1}$ if and only if ${\vec{v}}_i=0$ and ${\vec{v}}_{i+1}=1$. As a consequence, a T-walk $T_{\vec{v}}$ does not exist if and only if $\overline{\Gamma}({\vec{v}})$ is not an ideal of $P_\gamma$.
\end{proof}

\begin{proof}[Proof of Theorem~\ref{thm:texpansion}] We have shown previously shown that (see Theorem~\ref{thm:poset extension}) we have:
\begin{align*}x_{\gamma}=\displaystyle {x(T_0)} \, \rank(P_\gamma;xy).
\end{align*}

We can use the bijection from Theorem~\ref{thm:bijection} to rewrite this formula in the language of $T$-walks:
\[
x(T_0)\rank(P_\gamma;xy)=\sum_{I\in J(P_\gamma)} x(T_0) x(I) y(I)= \sum_{T_{\vec{v}} \in \mathrm{TW}(\gamma)}  x(T_{\vec{v}}) y(T_{\vec{v}}).
\]
\end{proof}

\section{Expansions via matrices}\label{sec:matrix}

\begin{table}
    \centering
    \begin{tabular}[t]{|c|c|c|}
\hline
         & Formula & Example \\
         \hline &&\\
    Addition     &\emph{(See Theorem~\ref{thm:oriented2})} &
   \\
   $\mathbf{P}\searrow\mathbf{Q}$-
 &\begin{tabular}{c}
 $\rmm_w((\mathbf{P}\searrow\mathbf{Q})\!\searrow)=\rmm_w(\mathbf{P}\!\searrow) \cdot \rmm_w(\mathbf{Q}\!\searrow) $\\
 $\drm_w((\mathbf{P}\searrow\mathbf{Q})\!\nearrow)=\rmm_w(\mathbf{P}\!\searrow) \cdot \drm_w(\mathbf{Q}\!\nearrow)$
 \end{tabular}& 
 \begin{tabular}{c}
 \begin{tikzpicture}[scale=.45]
\draw (0,0)--(1,1)--(3,-1);
\fill[white] (0,0) circle(.2) ;
\fill[red] (0,0) circle(.1) ;
\draw[red] (0,0) circle(.2);
\fill (1,1) circle(.1)  ;
\fill (2,0) circle(.1)  ;
\fill[blue] (3,-1) circle(.15)  ;
\draw[->, blue,dashed, thick] (3,-1)--(3.5,-1.5);
\node at (1.5,-1.5) {$\scriptstyle{\mathbf{P}\!\searrow}$\quad,};
\end{tikzpicture} \quad \begin{tikzpicture}[scale=.45]
\fill[white] (0,1) circle(.2) ;
\draw (0,-.5)--(1,.5);
\fill[white] (0,-.5) circle(.2) ;
\fill[red] (0,-.5) circle(.1) ;
\draw[red] (0,-.5) circle(.2);
\fill[blue] (1,.5) circle(.15)  ;
\draw[->, blue,dashed,thick] (1,.5)--(1.5,0);
\node at (.5,-1.5) {$\scriptstyle{\mathbf{Q}\!\searrow}$};
\end{tikzpicture}$\quad$ \raisebox{.2cm}{$\rightarrow$} $\quad$\begin{tikzpicture}[scale=.45]
\draw (0,0)--(1,1)--(3,-1)--(4,0);
\fill[white] (0,0) circle(.2) ;
\fill[red] (0,0) circle(.1) ;
\draw[red] (0,0) circle(.2);
\fill (1,1) circle(.1)  ;
\fill(1+2/3,1/3) circle(.1);
\fill (1+4/3,-1/3) circle(.1)  ;
\fill(3,-1) circle(.1);
\fill[blue] (4,0) circle(.15)  ;
\draw[->, blue,dashed, thick] (4,0)--(4.5,-.5);
\node at (2,-1.5) {$\scriptstyle{(\mathbf{P}\searrow\mathbf{Q})\!\searrow}$};
\end{tikzpicture} \end{tabular}\\ &&\\
    $\mathbf{P}\nearrow\mathbf{Q}$
        &\begin{tabular}{c}
 $\rmm_w((\mathbf{P}\nearrow\mathbf{Q})\!\searrow)=\drm_w(\mathbf{P}\!\nearrow) \cdot \rmm_w(\mathbf{Q}\!\searrow) $\\
 $\drm_w((\mathbf{P}\nearrow\mathbf{Q})\!\nearrow)=\drm_w(\mathbf{P}\!\nearrow) \cdot \drm_w(\mathbf{Q}\!\nearrow)$
 \end{tabular}&\begin{tabular}{c}
 \begin{tikzpicture}[scale=.45]
\draw (0,0)--(1,1)--(3,-1);
\fill[white] (0,0) circle(.2) ;
\fill[red] (0,0) circle(.1) ;
\draw[red] (0,0) circle(.2);
\fill (1,1) circle(.1)  ;
\fill (2,0) circle(.1)  ;
\fill[blue] (3,-1) circle(.15)  ;
\draw[->, blue,dashed, thick] (3,-1)--(3.5,-.5);
\node at (1.5,-1.5) {$\scriptstyle{\mathbf{P}\!\nearrow}$\quad,};
\end{tikzpicture}\quad \begin{tikzpicture}[scale=.45]
\fill[white] (0,1) circle(.2) ;
\draw (0,-.5)--(1,.5);
\fill[white] (0,-.5) circle(.2) ;
\fill[red] (0,-.5) circle(.1) ;
\draw[red] (0,-.5) circle(.2);
\fill[blue] (1,.5) circle(.15)  ;
\draw[->, blue,dashed,thick] (1,.5)--(1.5,0);
\node at (.5,-1.5) {$\scriptstyle{\mathbf{Q}\!\searrow}$};
\end{tikzpicture}$\quad$ \raisebox{.2cm}{$\rightarrow$} $\quad$\begin{tikzpicture}[scale=.45]
\draw (0,0)--(1,1)--(3,-1)--(5,1);
\fill[white] (0,0) circle(.2) ;
\fill[red] (0,0) circle(.1) ;
\draw[red] (0,0) circle(.2);
\fill (1,1) circle(.1)  ;
\fill(2,0) circle(.1);
\fill(3,-1) circle(.1);
\fill (4,0) circle(.1) ;
\fill[blue] (5,1) circle(.15)  ;
\draw[->, blue,dashed, thick] (5,1)--(5.5,.5);
\node at (3,-1.5) {$\scriptstyle{(\mathbf{P}\nearrow\mathbf{Q})\!\searrow}$};
\end{tikzpicture} \end{tabular}\\
\hline &&\\
        Loop &\emph{(See Theorem~\ref{thm:oriented2})}&\\
        $\close(\mathbf{P}\!\searrow)$ &$\rank( \close(\mathbf{P}\!\searrow);w)=\operatorname{tr}(\rmm_w({\mathbf{P}\!\searrow}))$ &\begin{tabular}{c}
        \begin{tikzpicture}[scale=.45]
\draw (0,0)--(1,1)--(3,-1);
\fill[white] (0,0) circle(.2) ;
\fill[red] (0,0) circle(.1) ;
\draw[red] (0,0) circle(.2);
\fill (1,1) circle(.1)  ;
\fill (2,0) circle(.1)  ;
\fill[blue] (3,-1) circle(.1)  ;
\draw[->, blue,dashed, thick] (3,-1)--(3.5,-1.5);
\node at (1.5,-1.5) {$\scriptstyle{\mathbf{P}\!\searrow}$};
\end{tikzpicture} $\quad$ \raisebox{.2cm}{$\rightarrow$} $\quad$   \begin{tikzpicture}[scale=.45]
\draw (0,-1)--(1,1)--(3,-.2)--(0,-1);
\fill (0,-1) circle(.1) ;
\fill (1,1) circle(.1)  ;
\fill (2,0.4) circle(.1)  ;
\fill (3,-.2) circle(.1);
\node at (1.5,-1.5) {$\scriptstyle{\close\,(\mathbf{P}\!\searrow)}$};
\end{tikzpicture}\end{tabular} \\ && \\
 $\close(\mathbf{P}\nearrow)$ &$\rank( \close(\mathbf{P}\nearrow);w)=\operatorname{tr}(\rmm_w({\mathbf{P}\!\nearrow}))$ &\begin{tabular}{c}
        \begin{tikzpicture}[scale=.45]
\draw (0,0)--(1,1)--(3,-1);
\fill[white] (0,0) circle(.2) ;
\fill[red] (0,0) circle(.1) ;
\draw[red] (0,0) circle(.2);
\fill (1,1) circle(.1)  ;
\fill (2,0) circle(.1)  ;
\fill[blue] (3,-1) circle(.1)  ;
\draw[->, blue,dashed, thick] (3,-1)--(3.5,-.5);
\node at (1.5,-1.5) {$\scriptstyle{\mathbf{P}\!\nearrow}$};
\end{tikzpicture} $\quad$ \raisebox{.2cm}{$\rightarrow$} $\quad$   \begin{tikzpicture}[scale=.45]
\draw (0,0)--(1,1)--(3,-1)--(0,0);
\fill (0,0) circle(.1) ;
\fill (1,1) circle(.1)  ;
\fill (2,0) circle(.1)  ;
\fill (3,-1) circle(.1);
\node at (1.5,-1.5) {$\scriptstyle{\close\,(\mathbf{P}\!\nearrow)}$};
\end{tikzpicture}\end{tabular} \\
        \hline  &&\\
        Source&&\\
        Loop&\emph{(See Theorem~\ref{theorem1})}&\\
             $\rhd(\mathbf{P}\!\searrow)$ &$\rmm_w( (\rhd(\mathbf{P}\!\searrow))\!\searrow) =\lloop_\searrow (\rmm_w({\mathbf{P}\!\searrow}))$ &\begin{tabular}{c}
        \begin{tikzpicture}[scale=.45]
\draw (0,0)--(1,1)--(3,-1);
\fill[white] (0,0) circle(.2) ;
\fill[red] (0,0) circle(.1) ;
\draw[red] (0,0) circle(.2);
\fill (1,1) circle(.1)  ;
\fill (2,0) circle(.1)  ;
\fill[blue] (3,-1) circle(.1)  ;
\draw[->, blue,dashed, thick] (3,-1)--(3.5,-1.5);
\node at (1.5,-1.5) {$\scriptstyle{\mathbf{P}\!\searrow}$};
\end{tikzpicture} $\quad$ \raisebox{.2cm}{$\rightarrow$} $\quad$   \begin{tikzpicture}[scale=.45]
\draw (0,-1)--(1,1)--(3,-.2)--(0,-1);
\fill (0,-1) circle(.1) ;
\fill (1,1) circle(.1)  ;
\fill (2,0.4) circle(.1)  ;
\fill[white] (3,-.2) circle(.2) ;
\fill[red] (3,-.2) circle(.1) ;
\draw[red] (3,-.2) circle(.2);
\draw[->, blue,dashed, thick] (3,-.2)--(3.8,-1);
\node at (1.5,-1.5) {$\scriptstyle{\rhd(\mathbf{P}\!\searrow)\!\searrow}$};
\end{tikzpicture}\end{tabular} \\ && \\
  $\rhd(\mathbf{P}\!\nearrow)$ &$\rmm_w(( \rhd(\mathbf{P}\!\nearrow))\!\searrow)=\lloop_\nearrow(
             \rmm_w({\mathbf{P}\!\nearrow}))$
  &\begin{tabular}{c}
        \begin{tikzpicture}[scale=.45]
\draw (0,0)--(1,1)--(3,-1);
\fill[white] (0,0) circle(.2) ;
\fill[red] (0,0) circle(.1) ;
\draw[red] (0,0) circle(.2);
\fill (1,1) circle(.1)  ;
\fill (2,0) circle(.1)  ;
\fill[blue] (3,-1) circle(.1)  ;
\draw[->, blue,dashed, thick] (3,-1)--(3.5,-.5);
\node at (1.5,-1.5) {$\scriptstyle{\mathbf{P}\!\nearrow}$};
\end{tikzpicture} $\quad$ \raisebox{.2cm}{$\rightarrow$} $\quad$   \begin{tikzpicture}[scale=.45]
\draw (0,0)--(1,1)--(3,-1)--(0,0);
\fill (0,0) circle(.1) ;
\fill (1,1) circle(.1)  ;
\fill (2,0) circle(.1)  ;
\fill[white] (3,-1) circle(.2) ;
\fill[red] (3,-1) circle(.1) ;
\draw[red] (3,-1) circle(.2);
\draw[->, blue,dashed, thick] (3,-1)--(3.9,-1.5);
\node at (1.5,-1.5) {$\scriptstyle{\rhd(\mathbf{P}\!\nearrow)\!\searrow}$};
\end{tikzpicture}\end{tabular} \\
        \hline  &&\\
        Target&&\\
        Loop&\emph{(See Theorem~\ref{theorem2})}&\\
             $\lhd(\mathbf{P}\!\searrow)$ &$\rmm_w( (\lhd(\mathbf{P}\!\searrow))\!\searrow)=\rloop(\rmm_w({\mathbf{P}\!\searrow}))$ &\begin{tabular}{c}
        \begin{tikzpicture}[scale=.45]
\draw (0,0)--(1,1)--(3,-1);
\fill[white] (0,0) circle(.2) ;
\fill[red] (0,0) circle(.1) ;
\draw[red] (0,0) circle(.2);
\fill (1,1) circle(.1)  ;
\fill (2,0) circle(.1)  ;
\fill[blue] (3,-1) circle(.1)  ;
\draw[->, blue,dashed, thick] (3,-1)--(3.5,-1.5);
\node at (1.5,-1.5) {$\scriptstyle{\mathbf{P}\!\searrow}$};
\end{tikzpicture} $\quad$ \raisebox{.2cm}{$\rightarrow$} $\quad$   \begin{tikzpicture}[scale=.45]
\draw (0,-1)--(1,1)--(3,-.2)--(0,-1);
\fill (3,-.2) circle(.1) ;
\fill (1,1) circle(.1)  ;
\fill (2,0.4) circle(.1)  ;
\fill[white] (0,-1) circle(.2) ;
\fill[red] (0,-1) circle(.1) ;
\draw[red] (0,-1) circle(.2);
\draw[->, blue,dashed, thick] (0,-1)--(-.5,-1.5);
\node at (1.5,-1.5) {$\scriptstyle{\lhd(\mathbf{P}\!\searrow)\!\searrow}$};
\end{tikzpicture}\end{tabular} \\ && \\
  $\lhd(\mathbf{P}\!\nearrow)$ &$\rmm_w( (\lhd(\mathbf{P}\!\nearrow))\!\searrow)=\rloop (\rmm_w({\mathbf{P}\!\nearrow}))$
  &\begin{tabular}{c}
        \begin{tikzpicture}[scale=.45]
\draw (0,0)--(1,1)--(3,-1);
\fill[white] (0,0) circle(.2) ;
\fill[red] (0,0) circle(.1) ;
\draw[red] (0,0) circle(.2);
\fill (1,1) circle(.1)  ;
\fill (2,0) circle(.1)  ;
\fill[blue] (3,-1) circle(.1)  ;
\draw[->, blue,dashed, thick] (3,-1)--(3.5,-.5);
\node at (1.5,-1.5) {$\scriptstyle{\mathbf{P}\!\nearrow}$};
\end{tikzpicture} $\quad$ \raisebox{.2cm}{$\rightarrow$} $\quad$   \begin{tikzpicture}[scale=.45]
\draw (0,0)--(1,1)--(3,-1)--(0,0);
\fill (0,0) circle(.1) ;
\fill (1,1) circle(.1)  ;
\fill (2,0) circle(.1)  ;
\fill (3,-1) circle(.1)  ;
\fill[white] (0,0) circle(.2) ;
\fill[red] (0,0) circle(.1) ;
\draw[red] (0,0) circle(.2);
\draw[->, blue,dashed, thick] (0,0)--(-.8,-.8);
\node at (1.5,-1.5) {$\scriptstyle{\lhd(\mathbf{P}\!\nearrow)\!\searrow}$};
\end{tikzpicture}\end{tabular} \\
\hline
    \end{tabular}
    \caption{Big Table of Oriented Poset Operations}
    \label{table:tableofmoves}
\end{table}

In this section, we will describe a method of calculating weight polynomials of posets with matrix multiplication.  
 An \definition{oriented poset} $\mathbf{P}=(P,x_L,x_R)$ consists of a poset $P$ with two specialized vertices $x_L$ and $x_R$ which can be thought as the left and right vertex, or alternatively the \definition{target vertex} $\target\,$  and the \definition{source vertex} $\rightarrow$. Oriented matrices actually behave like building blocks, in the sense that we connect the target of one with the arrow of another to build a larger structure. Also as building blocks, one can start with the smallest of pieces (a one element poset) to build a myriad of structures. One key difference is, there are two ways of an arrow vertex connecting with a target vertex. We can either connect via $x_R \succeq y_L$ ($x_R\nearrow y_L$) to get $\mathbf{P}\nearrow\mathbf{Q}$ or connect via $x_R \preceq y_L$ ($x_R\searrow y_L$) to get $\mathbf{P}\searrow\mathbf{Q}$.

On the weight polynomial level, we will try to streamline these operations into $2\times2$ matrices that keep track of the end points as well as the weight polynomials.  A \definition{weight matrix} of an oriented poset $\mathbf{P}$ can be thought of as an oriented matrix with an upwards or downwards pointing tail coming out of the arrow vertex. The two types of connections necessitate us to define two kinds of weight matrices.
\begin{itemize}
    \item \emph{Down pointing weight matrix of $\mathbf{P}$}: $$\displaystyle \rmm_w(\mathbf{P}\!\searrow):=\begin{bmatrix} \rank(\mathbf{P};w) & -\rank(\mathbf{P};q)|_{x_R\in I}\\ \rank(\mathbf{P};w)|_{x_L \notin I} & -\rank(\mathbf{P};w)|_{\substack{x_R\in I\\x_L \notin I}}
		\end{bmatrix}$$
    \item  \emph{Up pointing weight matrix of $\mathbf{P}$}: 
    $$	\displaystyle \drm_w(\mathbf{P}\!\nearrow):=\begin{bmatrix} \rank(\mathbf{P};w)|_{x_R\in I} & \rank(\mathbf{P};w)|_{x_R\notin I}\\ \rank(\mathbf{P};w)|_{\substack{x_R\in I\\x_L \notin I}} & \rank(\mathbf{P};w)|_{\substack{x_R\notin I\\x_L \notin I}}
			\end{bmatrix}$$
\end{itemize}
The entries are the weight polynomials of the ideals satisfying the given constraints.
We can easily go from one to the other via the \definition{tail flip} matrix $\tf=\begin{bmatrix}
				1&-1\\1&0
			\end{bmatrix}$:
		$$\drm_w(\mathbf{P}\!\nearrow)=\rmm_w(\mathbf{P}\!\searrow) \cdot \tf, \qquad \qquad \rmm_w(\mathbf{P}\!\searrow)=\drm_w(\mathbf{P}\!{\nearrow}) \cdot \tf^{-1}.$$

If we know a weight matrix of an oriented poset, we can easily read the rank polynomial off from it. For $\rmm_w(\mathbf{P}\!\searrow)$ it is given by the top left entry of the matrix and for $\drm_w(\mathbf{P}\!\nearrow)$ it is given by the sum of the entries in the top row. As we are primarily concerned with the weight polynomials for the expansion formulae calculation, it is sufficient to calculate either matrix. 
		
	\begin{figure}[ht]
	    \centering
\begin{tikzpicture}[scale=.7]
\draw (0,0)--(1.5,1)--(4.5,-1);
\fill (0,0) circle(.2) node[white] {$1$} ;
\fill (1.5,1) circle(.2) node[white] {$2$} ;
\fill (3,0) circle(.2) node[white] {$3$} ;
\fill (4.5,-1) circle(.2) node[white] {$4$} ;
\node at (2.25,-2.5) {Fence poset $F(1,2)$};
\end{tikzpicture}\qquad \qquad \begin{tikzpicture}[scale=.7]
\draw (0,0)--(1.5,1)--(4.5,-1);
\fill[white] (0,0) circle(.2) ;
\fill[red] (0,0) circle(.1) ;
\draw[red] (0,0) circle(.2);
\fill (1.5,1) circle(.1)  ;
\fill (3,0) circle(.1)  ;
\fill[blue] (4.5,-1) circle(.15)  ;
\draw[->, blue,dashed] (4.5,-1)--(5.5,-5/3);
\node at (2.25,-2.5) {$\mathbf{(1,2)}\!\searrow$};
\end{tikzpicture}
\qquad  \begin{tikzpicture}[scale=.7]
\draw (0,0)--(1.5,1)--(4.5,-1);
\fill[white] (0,0) circle(.2) ;
\fill[red] (0,0) circle(.1) ;
\draw[red] (0,0) circle(.2);
\fill (1.5,1) circle(.1)  ;
\fill (3,0) circle(.1)  ;
\fill[blue] (4.5,-1) circle(.15)  ;
\draw[->, blue,dashed] (4.5,-1)--(5.5,-1/3);
\node at (2.25,-2.5) {$\mathbf{(1,2)}\!\nearrow$};
\end{tikzpicture}
\caption{One can think of the fence poset for $(1,2)$ as an oriented poset by specifying left and right end points.}\label{fig:21example}
\end{figure}
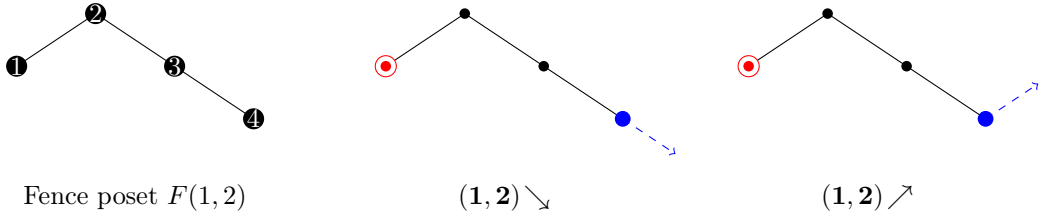

\begin{example} The fence poset of $(1,2)$ can be oriented by taking vertex $1$ as the target vertex, and vertex $4$ as the source vertex. There are two ways to add a tail to the source vertex (see Figure~\ref{fig:21example}) and the corresponding weight matrices are given below:

\begin{align*}
    \rmm_w(\mathbf{(1,2)}\!\searrow)&=\begin{bmatrix}
    \scriptstyle{1+w_1+w_4+w_1w_4+w_3w_4+w_1w_3w_4+w_1w_2w_3w_4} &  \scriptstyle{-w_4-w_1w_4-w_3w_4-w_1w_3w_4-w_1w_2w_3w_4}\\
     \scriptstyle{1+w_4+w_3w_4}&  \scriptstyle{-w_4-w_3w_4}
    \end{bmatrix},\\
        \drm_w(\mathbf{(1,2)}\!\nearrow)&=\begin{bmatrix}
    \scriptstyle{\qquad  w_4+w_1w_4+w_3w_4+w_1w_3w_4+w_1w_2w_3w_4} \qquad &  \hspace{20mm}\scriptstyle{1+w_1}\hspace{20 mm}\\
     \scriptstyle{w_4+w_3w_4}&  \scriptstyle{1}
    \end{bmatrix}.
\end{align*}

The weight polynomial of the fence poset $F(1,2)$ is equal to $1+w_1+w_4+w_1w_4+w_3w_4+w_1w_3w_4+w_1w_2w_3w_4$. It can be calculated by looking at top left entry of $\rmm_w(\mathbf{(1,2)}\!\searrow)$ or the sum of the entries in the top row of 
$\drm_w(\mathbf{(1,2)}\!\nearrow)$.
\end{example}

\begin{thm}\cite{O22}\label{thm:oriented1} The weight matrices for $\mathbf{P}\searrow\mathbf{Q}$ and $\mathbf{P}\nearrow\mathbf{Q}$ can be calculated from the weight matrices for $\mathbf{P}$ and $\mathbf{Q}$ by matrix multiplication as follows:
\begin{eqnarray*}\rmm_w((\mathbf{P}\searrow\mathbf{Q})\!\searrow)=\rmm_w(\mathbf{P}\!\searrow) \cdot \rmm_w(\mathbf{Q}\!\searrow), \qquad &\drm_w((\mathbf{P}\searrow\mathbf{Q})\!\nearrow)=\rmm_w(\mathbf{P}\!\searrow) \cdot \drm_w(\mathbf{Q}\!\nearrow),\\
\rmm_w((\mathbf{P}\nearrow\mathbf{Q})\!\searrow)=\drm_w(\mathbf{P}\!\nearrow) \cdot \rmm_w(\mathbf{Q}\!\searrow), \qquad &\drm_w((\mathbf{P}\nearrow\mathbf{Q})\!\nearrow)=\drm_w(\mathbf{P}\!\nearrow) \cdot \drm_w(\mathbf{Q}\!\nearrow).
\end{eqnarray*}
\end{thm}

This method allows us to construct all fence posets by combining posets with a single vertex by up and down steps. We will use the following notation for the weight matrices of the single vertex poset:
\begin{align}
\rmm_w(\bullet\!\searrow):=\mdo(w)=\mdom{w},\quad \qquad \drm_w(\bullet\!\nearrow):=\mup(w):=\mupm{w}.
\end{align}~\label{down-up}

These two matrices are sufficient for building all fence posets and calculating their weight polynomials. For circular fence posets, we will need an additional operation. Given an oriented poset, one may want to connect its source vertex to its target vertex, creating a loop. We denote the poset obtained thus from $\mathbf{P}$ by connecting the source to the target via $x_L\preceq x_R$ (respectively $x_L\succeq x_R$) by $\close(\mathbf{P}\!\searrow)$ (respectively $\close(\mathbf{P}\!\nearrow)$. One can simply do this by taking the trace of the corresponding matrix. 

\begin{thm}\cite{O22}\label{thm:oriented2} We can calculate the weight polynomials for $\close(\mathbf{P}\!\searrow)$ and $\close(\mathbf{P}\!\nearrow)$ by taking the traces of the corresponding weight matrices as follows:
\begin{eqnarray*}
     \rank( \close(\mathbf{P}\!\searrow);w)=\operatorname{tr}(\rmm_w({\mathbf{P}\!\searrow})) \qquad \rank( \close(\mathbf{P}\nearrow);w)=\operatorname{tr}(\rmm_w({\mathbf{P}\!\nearrow}))
\end{eqnarray*}
\end{thm}
This is a poset that is not oriented. If no more connections are necessary, as in the case of circular fence posets corresponding to the band graphs, it is enough to use the trace operation to calculate the corresponding weight polynomial:

\begin{align*}
    \rank( \close(\mathbf{P}\!\searrow);w)&=\operatorname{tr}(\rmm_w({\mathbf{P}\!\searrow})),\\
    \rank( \close(\mathbf{P}\!\nearrow);w)&=\operatorname{tr}(\rmm_w({\mathbf{P}\!\nearrow})).
\end{align*}

\begin{remark} The proofs of the above identities are given in \cite{O22} where oriented posets are originally defined. Although that work is concerned with rank polynomials and not weight polynomials, the same ideas apply. We will only include proofs for the new constructions given below.
\end{remark}

When we want to continue making connections to a loop  from the source vertex, we use the source loop operation. The \definition{source loop} of $\mathbf{P}\!\searrow$ is the oriented poset $(\close(\mathbf{P}\!\searrow),x_R,x_R)$. We denote it by $\rhd(\mathbf{P}\!\searrow)$. We can similarly define $\rhd(\mathbf{P}\!\nearrow)$. 

Calculating the weight matrix of a source loop requires a separate function as we are altering our source and target matrices.

\begin{align*}
    \lloop_\searrow\left( \begin{bmatrix}
    a&b\\c&d
    \end{bmatrix}\right):&=& \begin{bmatrix}
    a+d& b-d\\a+b&0
    \end{bmatrix} \qquad \qquad     \lloop_\nearrow\left( \begin{bmatrix}
    a&b\\c&d
    \end{bmatrix}\right):&=&  \begin{bmatrix}
    a+d& -a\\d&0
    \end{bmatrix}
\end{align*}
\vspace{.1cm}

\begin{thm}~\label{theorem1} The weight matrices correponding to $\rhd(\mathbf{P}\!\searrow)$ and $\rhd(\mathbf{P}\!\nearrow)$ can be calculated via the $\lloop$ functions as follows: 
\begin{align*}
    \lloop_\searrow(\rmm_w(\mathbf{P}\!\searrow))&= \rmm_w( (\rhd(\mathbf{P}\!\searrow))\!\searrow), \qquad 
    \lloop_\nearrow(\drm_w(\mathbf{P}\!\nearrow))= \rmm_w( (\rhd(\mathbf{P}\!\nearrow))\!\nearrow).
\end{align*}
\end{thm}
\begin{proof} Consider the entries of the matrix $\rmm_w( (\rhd(\mathbf{P}\!\searrow))\!\searrow)$. As the poset has an added connection $x_R\succeq x_L$, the top left corner is given by the trace of $\rmm_w(\mathbf{P}\!\searrow)$, which matches the top left entry of $\lloop_\searrow$. The top right corner is generated by all the ideals of this poset that contain the source vertex $x_R$, times $(-1)$. As the ideals that contain $x_R$ also contain $x_L$ by the relation $x_R\succeq x_L$, we have all the ideals of our original poset that contain $x_L$ minus the ideals that contain $x_L$ but not $x_R$, multiplied by $(-1)$. That is the top right entry minus the bottom right entry of $\rmm_w(\mathbf{P}\!\searrow)$. For the bottom left entry of $\rmm_w( (\rhd(\mathbf{P}\!\searrow))\!\searrow)$, we need the weight polynomial generated by the ideals containing the target vertex, which is also $x_R$. These remain unaffected by $x_R\succeq x_L$, so we can just read them of from the top right entry of $\rmm_w(\mathbf{P}\!\searrow)$, multiplied by $(-1)$. Finally, as the source and target vertices are the same, the bottom right entry must be $0$. The entries of $\rmm_w( (\rhd(\mathbf{P}\!\searrow))\!\searrow)$ can be similarly computed.
\end{proof}

We can also define the target loop operation, where any further connections are made using the target (left) vertex. The \definition{target loop} of $\mathbf{P}\!\searrow$ is the oriented poset $(\close(\mathbf{P}\!\searrow),x_L,x_L)$, denoted by $\lhd(\mathbf{P}\!\searrow)$. The target loop  $\lhd(\mathbf{P}\!\nearrow)$ is similarly defined. 

For the target loop, the direction of the loop does not change the function, so one function is sufficent.

\begin{align*}
    \rloop\left( \begin{bmatrix}
    a&b\\c&d
    \end{bmatrix}\right): =& \begin{bmatrix}
    a+d& c-a\\c+d&0
    \end{bmatrix}
\end{align*}

\begin{thm}~\label{theorem2} The weight matrices correponding to $\lhd(\mathbf{P}\!\searrow)$ and $\lhd(\mathbf{P}\!\nearrow)$ can be calculated via the $\lloop$ functions as follows: 
\begin{align*}
    \rloop(\rmm_w(\mathbf{P}\!\searrow))&= \rmm_w( (\lhd(\mathbf{P}\!\searrow))\!\searrow), \qquad 
    \rloop(\drm_w(\mathbf{P}\!\nearrow))= \rmm_w( (\lhd(\mathbf{P}\!\nearrow))\!\nearrow).
\end{align*}
\end{thm}
\begin{proof} One can calculate coordinate by coordinate as in the source loop case. 
\end{proof}

We can use the above theorems to get matrix formulas for a variety of arcs. Here we show how to use matrices to calculate the expansion formula for the arc from Figure~\ref{fig:twalks}. See Examples~\ref{ex:twalksviatwalks} and \ref{ex:twalksviaposets} for the same calculation using $T$-walks and poset ideals. See the Appendix~\ref{appendix} for calculations for several examples of differing complexity. See Table~\ref{table:tableofmoves} for a list of all the operations listed above, their formulas and examples. 

\begin{example}\label{ex:twalksviamatrices}
We can form the poset from Figure~\ref{fig:twalks4} by one down step, one up step and another down step as follows.
\begin{figure}[H]
    \centering
    \scalebox{.8}{
\begin{tikzpicture}
\draw[->, blue, dashed, thick] (1.5,-1)--(2.5,-1/3);
\draw[->, blue, dashed, thick] (3,0)--(4,-2/3);
\draw[->, blue, dashed, thick] (4.5,-1)--(5.5,-5/3);

\fill (1.5,-1) circle(.2) node[white] {$1$} ;
\draw (1.6,-1.5) node{$\displaystyle w_1$} ;
\draw (3.1,-.5) node{$\displaystyle w_2$} ;
\fill (4.5,-1) circle(.2) node[white] {$3$} ;
\draw (4.6,-1.5) node{$\displaystyle w_3$} ;
\fill[white] (1.5,-1) circle(.2) ;
\fill[red] (1.5,-1) circle(.1) ;
\draw[red] (1.5,-1) circle(.2);
\fill[white] (3,0) circle(.2) ;
\fill[red] (3,0) circle(.1) ;
\draw[red] (3,0) circle(.2);
\fill[white] (4.5,-1) circle(.2) ;
\fill[red] (4.5,-1) circle(.1) ;
\draw[red] (4.5,-1) circle(.2);
\end{tikzpicture} \raisebox{15mm}{$\qquad$ \scalebox{2}{$\Rightarrow$} $\qquad$}
\begin{tikzpicture}
\draw (1.5,-1)-- (3,0)--(4.5,-1);
\draw[->, blue, dashed,thick] (4.5,-1)--(5.5,-5/3);
\fill (1.5,-1) circle(.1);
\draw (1.6,-1.5) node{$\displaystyle w_2$} ;
\fill (3,0) circle(.1) ;
\draw (3.1,-.5) node{$\displaystyle w_3$} ;
\fill[blue] (4.5,-1) circle(.15);
\draw (4.6,-1.5) node{$\displaystyle w_4$} ;
\fill[white] (1.5,-1) circle(.2) ;
\fill[red] (1.5,-1) circle(.1) ;
\draw[red] (1.5,-1)circle(.2);
\end{tikzpicture}}
\end{figure}
The corresponding matrix formula is:
\begin{align*}
\displaystyle & \mup(w_1)\mdo(w_2)\mdo(w_3)=
    \mupm{w_1}\mdom{w_2}\mdom{w_3}\\
    &=\begin{bmatrix}
        1+w_1+w_3+w_1w_3+w_1w_2w_3 & -w_3 -w_1w_3 - w_1w_2w_3\\ 1 + w_3& -w_3
    \end{bmatrix}.
\end{align*}
The top left hand corner gives us the weight polynomial. Then we plug in the labels from Figure~\ref{fig:twalks4} and multiply with $x(T_0)$ like in Example~\ref{ex:twalksviaposets} to get the expansion formula as desired.
\end{example}

\section{The case of self-folded triangles}\label{sec:selffold}

In this section, we will consider the cases surfaces where the initial triangulation has a self-folded triangle. While the general idea remains the same, there are some intricacies in the self-folded case that one needs to pay attention to.

Let $\gamma$ on be an arc on a surface where the initial triangulation has a self-folded triangle. We can think of a self-folded triangle as a degenerate case where a \definition{radius} arc is considered as two of the three sides of the triangle. To compute the correct $T$-walks expansion, we will need to seperate this radius arc $r$ into two parts $r$ and $\underline{r}$, and we assume $\underline{r}$ is following $r$ in the counterclockwise direction, as seen below in Figure~\ref{fig:opening}.

\begin{figure}[H]
\includegraphics[width=7.6cm]{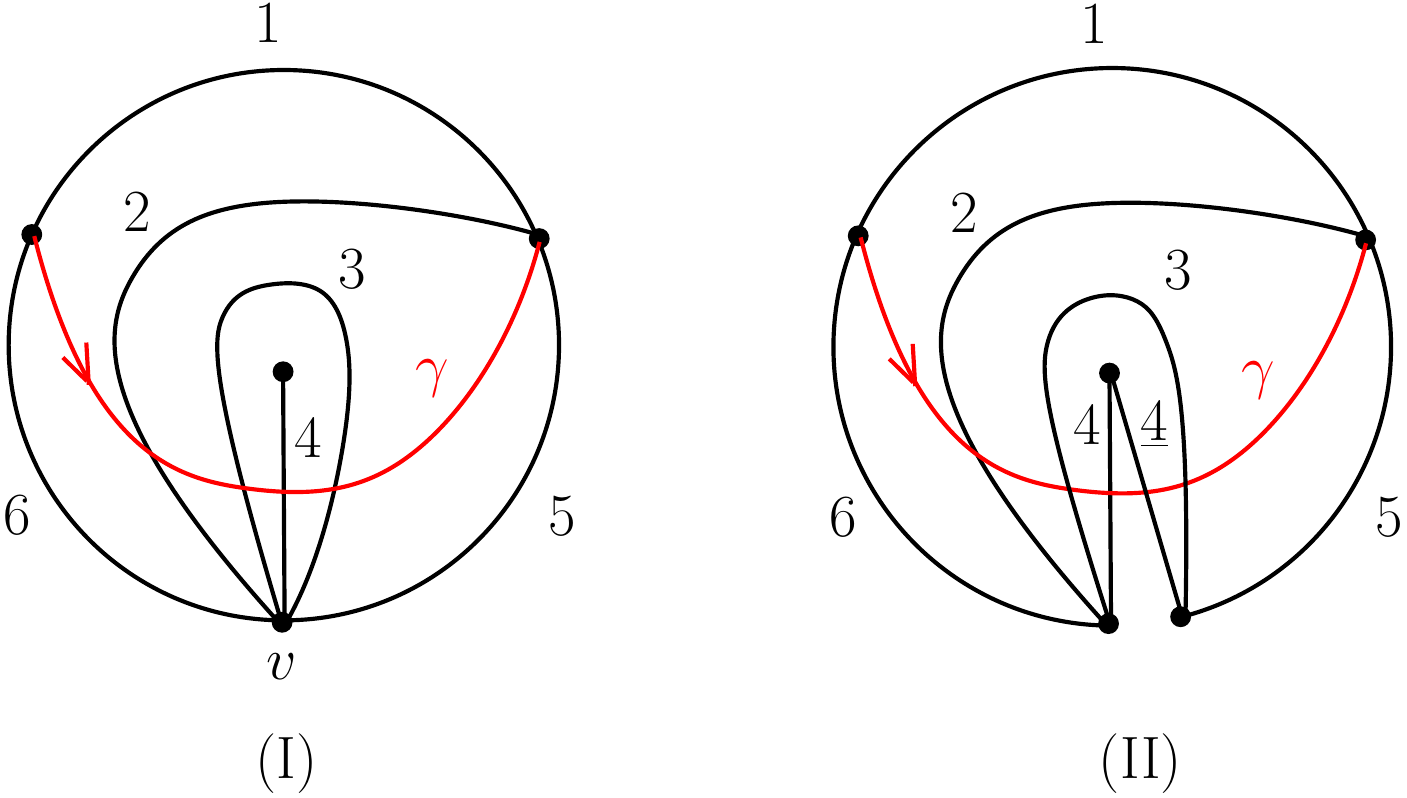}
\caption{Illustration of opening of the radius.} \label{fig:opening}
\end{figure}

More precisely, we seperate the vertex $v$ where the arcs of our self-folded triangle connect into two vertices so that we obtain an actual triangle. (This perspective is often used for self-folded case to resove ambiguities.) Even though our arc $\gamma$ now crosses both $r$ and $\underline{r}$ in this visualisation, we will only count one of these crossings when calculating the associated $T$-walks.

\begin{remark}~\label{rmk:res} If the orientation of the radius in the original surface is $0$, then we take the crossing $\alpha_i$ to be $r$. Otherwise, if the orientation is $1$, we use $\underline{r}$. Note that this only restricts the $\alpha$'s in the associated $T$-walk. For the $\beta$'s, we are allowed to use $r$ or $\underline{r}$ in any orientation. 
\end{remark}

\begin{example}
    For the arc $\gamma$ in Figure~\ref{fig:opening}, we get five $T$-walks, listed below with their corresponding orientation vectors. 
\begin{itemize}
    \item For $\vec{v_1}=(0,0,0,0)$, we get the $T$-walk  $(6,2,2,3,3,4,\underline{4},3,2)$,
    \item For $\vec{v_2}=(0,0,0,1)$, we get the $T$-walk   $(6,2,2,3,3,4,4,3,5)$,
    \item For $\vec{v_3}=(0,0,1,1)$, we get the $T$-walk  $(6,2,2,3,\underline{4},\underline{4},3,3,5)$,
    \item For $\vec{v_4}=(0,1,1,1)$, we get the $T$-walk  $(6,2,5,3,4,\underline{4},3,3,5)$,
    \item For $\vec{v_5}=(1,1,1,1)$, we get the $T$-walk   $(1,2,3,3,4,\underline{4},3,3,5)$.
\end{itemize}
The other orientation vectors ($9$ other possible choices) do not lead to $T$-walks; some will not respect the restriction in Remark~\ref{rmk:res}.

\begin{figure}[H]
\includegraphics[width=2.7cm]{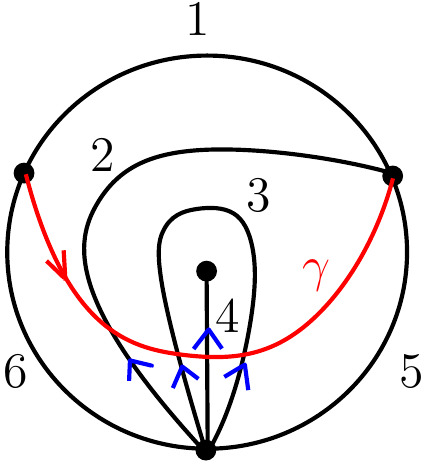}
\caption{Minimal orientation $\vec{v_1}$ shown for this surface.}
\end{figure}

\end{example}

The weight in $x$ variables for the $T$-walks are defined exactly as in the non self-folded case, where both $r$ and $\underline{r}$ contribute $x_r$. For the $y$ variable of $r$ or $\underline{r}$, there is a slight change: we use  $y_r/y_l$ where $l$ is the loop arc around the radius $r$. 

\begin{example}\label{ex:selffold}
    The $x$ and $y$ values for the five $T$-walks for $\gamma$ from Figure~\ref{fig:opening} are given below. 
\begin{itemize}
    \item $x_{v_1}=\frac{x_2x_6}{x_3}$, $\quad y_{v_1}=1$,
    \item $x_{v_2}=\frac{x_5x_6}{x_3}$, $\quad y_{v_2}=y_3$,
    \item $x_{v_3}=\frac{x_5x_6}{x_3}$, $\quad y_{v_3}=y_4$,
    \item $x_{v_4}=\frac{x_5^2x_6}{x_2x_3}$, $\quad y_{v_4}=y_3y_4$.
    \item $x_{v_5}=\frac{x_1x_5}{x_2}$, $\quad y_{v_5}=y_2y_3y_4$.
\end{itemize}

By Theorem~\ref{thm:texpansion} we have:
\[
\displaystyle x_\gamma= \frac{x_2x_6}{x_3}+ \frac{x_5x_6y}{x_3}y_3+ \frac{x_5x_6}{x_3}y_4 +\frac{x_5^2x_6}{x_2x_3}y_3y_4+\frac{x_1x_5}{x_2} y_2y_3y_4.
\]
\end{example}

We now calculate this expansion using the labeled poset method.

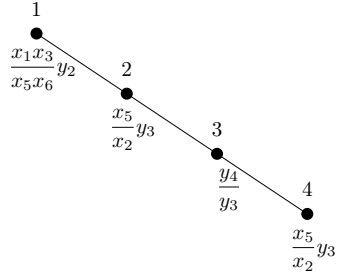
\begin{figure}[H]
 \begin{center}\scalebox{.8}{
\begin{tikzpicture}
\draw (0,0)--(4.5,-3); 
\fill (0,0) circle(.1) node[above,yshift=.17cm] {1};
\draw (.1,-.6) node{$\displaystyle \frac{x_1 x_3}{x_{5}x_{6}} y_2$} ;
\fill (1.5,-1) circle(.1) node[above,yshift=.17cm] {2} ;
\draw (1.6,-1.6) node{$\displaystyle \frac{x_{5}}{x_2}y_3$};
\fill (3,-2) circle(.1) node[above,yshift=.17cm]{3} ;
\draw (3.2,-2.6) node{$\displaystyle \frac{y_4}{y_3}$} ;
\fill (4.5,-3) circle(.1) node[above,yshift=.17cm] {4};
\draw (4.6,-3.6) node{$\displaystyle \frac{x_{5}}{x_2}y_3$};
\end{tikzpicture}}
\end{center}
\caption{Labeled poset for the arc from Figure~\ref{fig:opening}.}\label{fig:selffoldposet}
\end{figure}

When using the poset model for the calculations, the same formula can be used for the $x$ variables of radii and the loops:
$x(a_i):=\frac{x_{\alpha} x_{\beta}}{x_{\sigma} x_{\epsilon}}$
where $\alpha$ and $\beta$ are the labels of the arcs that follow the arc $e_i$ counterclockwise and $\sigma$ and $\epsilon$ are the arcs that follow $e_i$ clockwise. In fact, this formula becomes simpler, as radius arcs always cancel out. Also, if $r$ is a radius, we have full cancellation, giving us $x(r)=1$. For the $y$ variables, we again use $y_{e_i}$ unless $e_i=r$ is a radius. We set $y(r)=y_r/y_l$ where $l$ is the loop arc around the radius $r$, just like we did for the $T$-walks.

\begin{example} The labeled poset for the arc from Figure~\ref{fig:opening} is depicted in Figure~\ref{fig:selffoldposet}. We can calculate the weight polynomial directly, or looking at the top lefthand entry of the rank matrix:
\begin{align*}
&\mdom{w_1}\mdom{w_2}\mdom{w_3}\mdom{w_4}=\\
&\begin{bmatrix}
    1+w_4+w_3w_4+w_2w_3w_4+w_1w_2w_3w_4 & - w_4-w_3w_4-w_2w_3w_4-w_1w_2w_3w_4\\ 1+w_4+w_3w_4+w_2w_3w_4& - w_4-w_3w_4-w_2w_3w_4
\end{bmatrix}
\end{align*}
As we already computed the minimal $T$-walk $(6,2,2,3,3,4,\underline{4},3,2)$ to be $x_2x_6/{x_3}$ in Example~\ref{ex:selffold} above, we can just plug in the labels to get the cluster expansion from the weight polynomial of the poset.

\begin{align*}x_{\gamma}=\displaystyle {x(T_0)} \, \rank(P_\gamma;xy)&=\frac{x_2x_6}{x_3}\left( 1+\frac{x_5}{x_2}y_3+\frac{x_5}{x_2}y_4+\frac{x_5^2}{x_2^2}y_3y_4+\frac{x_1x_3x_5}{x_2^2x_6}y_2y_3y_4\right)\\
&=\frac{x_2x_6}{x_3}+\frac{x_5x_6}{x_3}y_3+\frac{x_5x_6}{x_3}y_4+\frac{x_5^2x_6}{x_2x_3}y_3y_4+\frac{x_1x_5}{x_2}y_2y_3y_4.
\end{align*}
\end{example}

  

\appendix

\section{Matrix Calculation Examples}\label{appendix}

In this section, we illustrate the the matrix calculation techniques described above through different kinds of examples that come up in the setting of cluster algebras. We show how posets corresponding to arcs can be built step by step via matrices, and how the corresponding expansion formula can be calculated in each case. We have chosen to use some examples that already exist in the literature (See \cite{wilson}) for easy comparison. Because of that, we use the following variation of Equation~\ref{eqn:main} from Remark~\ref{remark:clusterversion} in this section: \begin{align*} 
x_{\gamma}= \displaystyle \frac{x(M_{-})}{cross(T,\gamma)} 
 \rank(P_{\gamma};xy).
\end{align*} 

\subsection{Fence posets and snake graphs}

If the fence poset corresponding to our arc has no loops, we can build it using only down and up weight matrices as in Equation~\ref{down-up}, where the direction of the matrix for vertex $i$ is determined by the direction of the connection between $i$ and $i+1$. We will always use the down weight matrix for the final vertex as a convention.

Consider the arc from Example~\ref{ex:noloop2}, accompanied by Figures \ref{fig:sch10}, \ref{fig:min} and \ref{fig:minimalscompared}. 
The underlying fence can be built by combining up and down steps as follows:
\begin{center} \scalebox{.8}{
\begin{tikzpicture}
\draw[->, blue, dashed, thick] (0,0)--(1,-2/3);
\draw[->, blue, dashed, thick] (1.5,-1)--(2.5,-1/3);
\draw[->, blue, dashed, thick] (3,0)--(4,2/3);
\draw[->, blue, dashed, thick] (4.5,1)--(5.5,1/3);
\draw[->, blue, dashed, thick] (6,0)--(7,2/3);
\draw[->, blue, dashed, thick] (7.5,1)--(8.5,5/3);
\draw[->, blue, dashed, thick] (9,2)--(10,4/3);
\draw[->, blue, dashed, thick] (10.5,1)--(11.5,1/3);
\draw[->, blue, dashed, thick] (12,0)--(13,-2/3);
\fill (0,0) circle(.2) node[white] {$1$} ;
\draw (.1,-.5) node{$\displaystyle w_1$} ;%
\fill (1.5,-1) circle(.2) node[white] {$2$} ;
\draw (1.6,-1.5) node{$\displaystyle w_2$} ;
\fill (3,0) circle(.2) node[white] {$3$} ;
\draw (3.1,-.5) node{$\displaystyle w_3$} ;
\fill (4.5,1) circle(.2) node[white] {$4$} ;
\draw (4.6,.5) node{$\displaystyle w_4$} ;
\fill (6,0) circle(.2) node[white] {$5$} ;
\draw (6.1,-.5) node{$\displaystyle w_5$} ;
\fill (7.5,1) circle(.2) node[white] {$6$} ;
\draw (7.6,.5) node{$\displaystyle w_6$} ;
\fill (9,2) circle(.2) node[white] {$7$} ;
\draw (9.1,1.5) node{$\displaystyle w_7$} ;
\fill (10.5,1) circle(.2) node[white] {$8$} ;
\draw (10.6,.5) node{$\displaystyle w_8$} ;
\fill (12,0) circle(.2) node[white] {$9$} ;
\draw (12.1,-.5) node{$\displaystyle w_9$} ;
\fill[white] (0,0) circle(.2) ;
\fill[red] (0,0) circle(.1) ;
\draw[red] (0,0) circle(.2);
\fill[white] (1.5,-1) circle(.2) ;
\fill[red] (1.5,-1) circle(.1) ;
\draw[red] (1.5,-1) circle(.2);
\fill[white] (3,0) circle(.2) ;
\fill[red] (3,0) circle(.1) ;
\draw[red] (3,0) circle(.2);
\fill[white] (4.5,1) circle(.2) ;
\fill[red] (4.5,1) circle(.1) ;
\draw[red] (4.5,1) circle(.2);
\fill[white] (6,0) circle(.2) ;
\fill[red] (6,0) circle(.1) ;
\draw[red] (6,0) circle(.2);
\fill[white] (7.5,1) circle(.2) ;
\fill[red] (7.5,1) circle(.1) ;
\draw[red] (7.5,1) circle(.2);
\fill[white] (9,2) circle(.2) ;
\fill[red] (9,2) circle(.1) ;
\draw[red] (9,2) circle(.2);
\fill[white] (10.5,1) circle(.2) ;
\fill[red] (10.5,1) circle(.1) ;
\draw[red] (10.5,1) circle(.2);
\fill[white] (12,0) circle(.2) ;
\fill[red] (12,0) circle(.1) ;
\draw[red] (12,0) circle(.2);
\end{tikzpicture}}
\end{center}

The corresponding matrix is given by: $\mdo(w_1)\mup(w_2)\mup(w_3)\mdo(w_4)\mup(w_5)\mup(w_6)\mdo(w_7)\mdo(w_8)\mdo(w_9)$. This can be expressed as a matrix product as follows:

\vspace{2mm}

\scalebox{.85}{$\mdom{w_1}\mupm{w_2}\mupm{w_3}\mdom{w_4}\mupm{w_5}\mupm{w_6}\mdom{w_7}\mdom{w_8}\mdom{w_9}$.}

\vspace{2mm}

The upper left entry gives us the weight polynomial $\rank(P_\gamma;w)$.

\begin{dmath*}
\rank(P_\gamma;w)=
1+ w_2 + w_5 + w_9 + w_1w_2 + w_2w_3 + w_2w_5 + w_5w_6 + w_2w_9 + w_5w_9 + w_8w_9 + w_1w_2w_3 + w_1w_2w_5 + w_2w_3w_5 + w_2w_5w_6 + w_1w_2w_9 + w_2w_3w_9 + w_2w_5w_9 + w_5w_6w_9 + w_2w_8w_9 + w_5w_8w_9  + w_1w_2w_3w_5 + w_2w_3w_4w_5 + w_1w_2w_5w_6 + w_2w_3w_5w_6 + w_1w_2w_3w_9 + w_1w_2w_5w_9 + w_2w_3w_5w_9 + w_2w_5w_6w_9 + w_1w_2w_8w_9 + w_2w_3w_8w_9 + w_2w_5w_8w_9 + w_5w_6w_8w_9 + w_1w_2w_3w_4w_5 + w_1w_2w_3w_5w_6 + w_2w_3w_4w_5w_6 + w_1w_2w_3w_5w_9 + w_2w_3w_4w_5w_9 + w_1w_2w_5w_6w_9 + w_2w_3w_5w_6w_9 + w_1w_2w_3w_8w_9 + w_1w_2w_5w_8w_9 + w_2w_3w_5w_8w_9 + w_2w_5w_6w_8w_9 + w_5w_6w_7w_8w_9 + w_1w_2w_3w_4w_5w_6 + w_1w_2w_3w_4w_5w_9 + w_1w_2w_3w_5w_6w_9 + w_2w_3w_4w_5w_6w_9 + w_1w_2w_3w_5w_8w_9 + w_2w_3w_4w_5w_8w_9 + w_1w_2w_5w_6w_8w_9 + w_2w_3w_5w_6w_8w_9 + w_2w_5w_6w_7w_8w_9  + w_1w_2w_3w_4w_5w_6w_9 + w_1w_2w_3w_4w_5w_8w_9 + w_1w_2w_3w_5w_6w_8w_9 + w_2w_3w_4w_5w_6w_8w_9 + w_1w_2w_5w_6w_7w_8w_9 + w_2w_3w_5w_6w_7w_8w_9+ w_1w_2w_3w_4w_5w_6w_8w_9 + w_1w_2w_3w_5w_6w_7w_8w_9 + w_2w_3w_4w_5w_6w_7w_8w_9
+w_1w_2w_3w_4w_5w_6w_7w_8w_9 
\end{dmath*}

We now evaluate the weight polynomial with the labels from Example~\ref{ex:noloop2} in the formula below to get $64$ term expansion formula. 

 \begin{align*}
  x_{\gamma}=\displaystyle \frac{x_1x_3x_4^2x_6x_7^2x_8x_{10}}{x_1x_2x_3x_4x_5x_6x_7x_8x_9x_{10}} \rank(P_\gamma;xy)=\displaystyle \frac{x_4x_7}{x_2x_5x_9} \rank(P_\gamma;xy).
\end{align*}

Note that while looking at larger triangulations increases the number of matchings exponentially, in the matrix expansion it just corresponds to adding one more matrix per triangle. 

\subsection{Circular fence posets and band graphs} Band graphs are obtained by considering the closed curves on the surfaces. We can create their associated snake graphs by identifying a boundary edge of the first tile with a boundary edge of the last tile, where both edges have the same sign. 

Consider the following is an example of a closed arc on a surface of an annulus with initial triangulation with five arcs, depicted in Figure~\ref{fig:band}. It shows the corresponding band graph with the fence poset drawn.  

\begin{figure}[H]
\includegraphics[width=13cm]{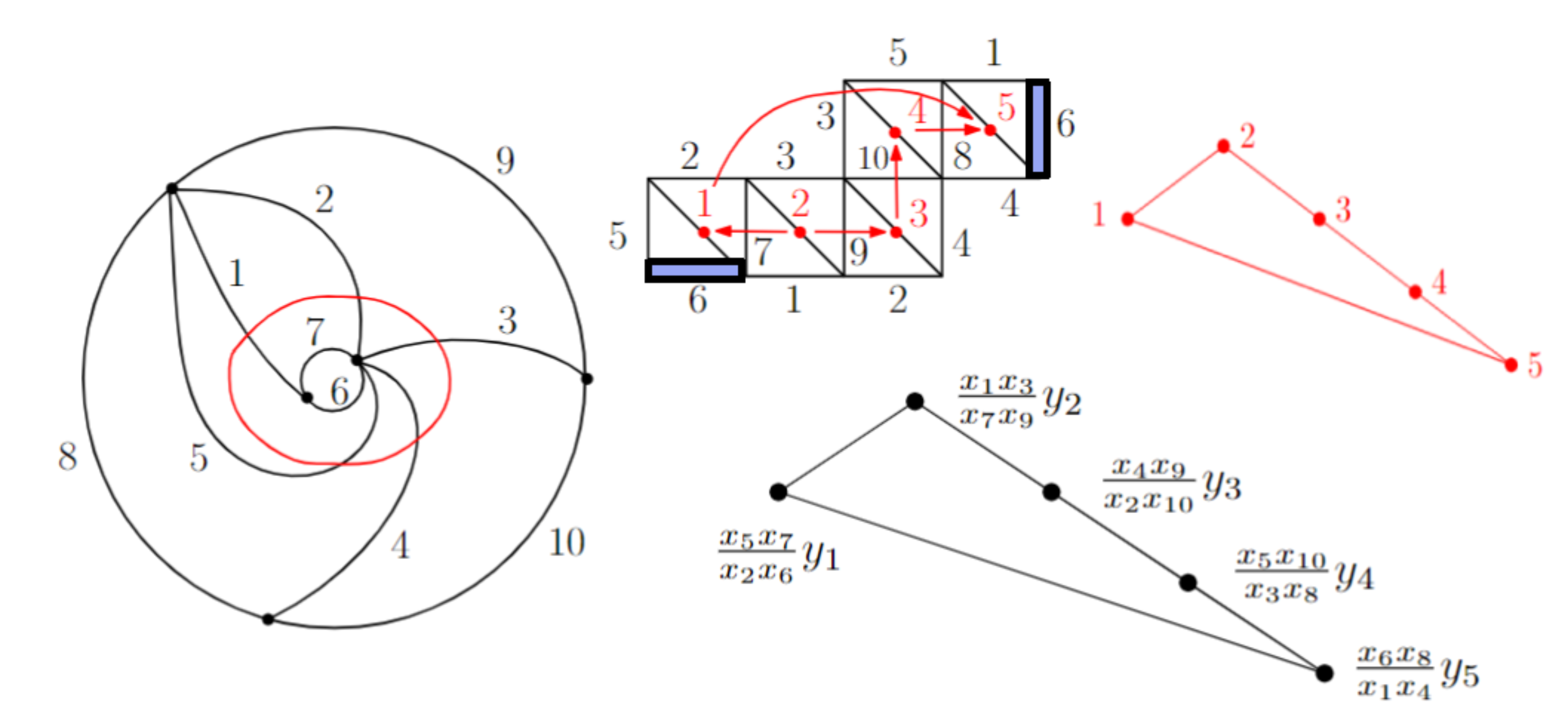}
\caption{A closed arc with the corresponding band graph, poset and labeled poset. The marked edges in the band graph are glued.} \label{fig:band}
\end{figure}

Note that it is preferable to write the poset for the closed arc directly from the surface. We build the corresponding oriented poset by combining up and down steps, and than taking the trace to connect the two ends.

\begin{center} \scalebox{.8}{
\begin{tikzpicture}
\draw[->, blue, dashed, thick] (7.5*.8,1)--(8.5*.8,5/3);
\draw[->, blue, dashed, thick] (9*.8,2)--(10*.8,4/3);
\draw[->, blue, dashed, thick] (10.5*.8,1)--(11.5*.8,1/3);
\draw[->, blue, dashed, thick] (12*.8,0)--(13*.8,-2/3);
\draw[->, blue, dashed, thick] (13.5*.8,-1)--(14.5*.8,-1/3);
\fill (7.5*.8,1) circle(.2) node[white] {$6$} ;
\draw (7.6*.8,.5) node{$\displaystyle w_1$} ;
\fill (9*.8,2) circle(.2) node[white] {$7$} ;
\draw (9.1*.8,1.5) node{$\displaystyle w_2$} ;
\fill (10.5*.8,1) circle(.2) node[white] {$8$} ;
\draw (10.6*.8,.5) node{$\displaystyle w_3$} ;
\fill (12*.8,0) circle(.2) node[white] {$9$} ;
\draw (12.1*.8,-.5) node{$\displaystyle w_4$} ;
\fill (12*.8,0) circle(.2) node[white] {$9$} ;
\draw (13.6*.8,-1.5) node{$\displaystyle w_5$} ;
\fill[white] (7.5*.8,1) circle(.2) ;
\fill[red] (7.5*.8,1) circle(.1) ;
\draw[red] (7.5*.8,1) circle(.2);
\fill[white] (9*.8,2) circle(.2) ;
\fill[red] (9*.8,2) circle(.1) ;
\draw[red] (9*.8,2) circle(.2);
\fill[white] (10.5*.8,1) circle(.2) ;
\fill[red] (10.5*.8,1) circle(.1) ;
\draw[red] (10.5*.8,1) circle(.2);
\fill[white] (12*.8,0) circle(.2) ;
\fill[red] (12*.8,0) circle(.1) ;
\draw[red] (12*.8,0) circle(.2);
\fill[white] (13.5*.8,-1) circle(.2) ;
\fill[red] (13.5*.8,-1) circle(.1) ;
\draw[red] (13.5*.8,-1) circle(.2);
\end{tikzpicture}
 \raisebox{15mm}{ \raisebox{.4cm}{\scalebox{2}{$\Rightarrow$}}}
\begin{tikzpicture}
\draw (7.5*.7,1)--(9*.7,2)--(13.5*.7,-1);
\draw[->, blue, dashed, thick] (13.5*.7,-1)--(14.5*.7,-1/3);
\fill (7.5*.7,1) circle(.1) ;
\draw (7.6*.7,.5) node{$\displaystyle w_1$} ;
\fill (9*.7,2) circle(.1) ;
\draw (9.1*.7,1.5) node{$\displaystyle w_2$} ;
\fill (10.5*.7,1) circle(.1) ;
\draw (10.6*.7,.5) node{$\displaystyle w_3$} ;
\fill (12*.7,0) circle(.1) ;
\draw (12.1*.7,-.5) node{$\displaystyle w_4$} ;
\fill[blue] (13.5*.7,-1) circle(.15);
\draw (13.6*.7,-1.5) node{$\displaystyle w_5$} ;
\fill[white] (7.5*.7,1) circle(.2) ;
\fill[red] (7.5*.7,1) circle(.1) ;
\draw[red] (7.5*.7,1) circle(.2);
\end{tikzpicture}
 \raisebox{15mm}{\raisebox{.4cm}{\scalebox{2}{$\Rightarrow$}} }
\begin{tikzpicture}
\draw (7.5*.7,1)--(9*.7,2)--(13.5*.7,-1)--(7.5*.7,1);
\fill (7.5*.7,1) circle(.1);
\draw (7.6*.7,.5) node{$\displaystyle w_1$};
\fill (9*.7,2) circle(.1);
\draw (9.1*.7,1.5) node{$\displaystyle w_2$};
\fill (10.5*.7,1) circle(.1);
\draw (10.6*.7,.5) node{$\displaystyle w_3$};
\fill (12*.7,0) circle(.1);
\draw (12.1*.7,-.5) node[below,yshift=.4cm]{$\displaystyle w_4$};
\fill (13.5*.7,-1) circle(.1);
\draw (13.6*.7,-1.5) node{$\displaystyle w_5$};
\end{tikzpicture}}
\end{center}

We get the following formula for the weight polynomial:
\begin{align*}
 \rank(P_\gamma;xy)&=   \operatorname{tr}(\mup(w_1)\mdo(w_2)\mdo(w_3)\mdo(w_4)\mup(w_5))\\
    &=\operatorname{tr}\left(\mupm{w_1}\mdom{w_2}\mdom{w_3}\mdom{w_4}\mupm{w_5}\right)\\
    &=1+w_5+w_1w_5+w_4w_5+w_1w_4w_5+w_3w_4w_5+w_1w_3w_4w_5+w_1w_2w_3w_4w_5.
\end{align*}

As the minimal matching gives $x_1x_2^2x_3x_4$, the expansion formula is as follows:
\begin{dmath*}x_{\gamma}=\displaystyle {\frac{x_1x_2^2x_3x_4}{x_1x_2x_3x_4x_5}} \rank(P_\gamma;xy)=\frac{x_2}{x_5}+\frac{x_2x_6x_8}{x_1x_4x_5}y_5+\frac{x_7x_8}{x_1x_4}y_1y_5+\frac{x_2x_5x_6x_{10}}{x_1x_3x_4}y_4y_5+\frac{x_5^2x_7x_{10}}{x_1x_3x_4}y_1y_4y_5+\frac{x_4x_5x_6x_{9}}{x_1x_3x_4}y_3y_4y_5+\frac{x_5^2x_7x_{9}}{x_1x_2x_3}y_1y_3y_4y_5+\frac{x_5^2}{x_2}y_1y_2y_3y_4y_5.
\end{dmath*}
\subsection{Fence posets with added loops and loop graphs}

To build the fence poset corresponding to a loop graph, we will need the above operations as well as the source and target loop operations whenever they are necessary. We will do step by step calculations for three examples: one with a single notched arc, and two with the double notched arcs.

\begin{example}\label{ex:single} Consider the arc given in Figure~\ref{singlenotched}. This example can also be found in \cite{wilson}.

\begin{figure}[H]
\includegraphics[width=13cm]{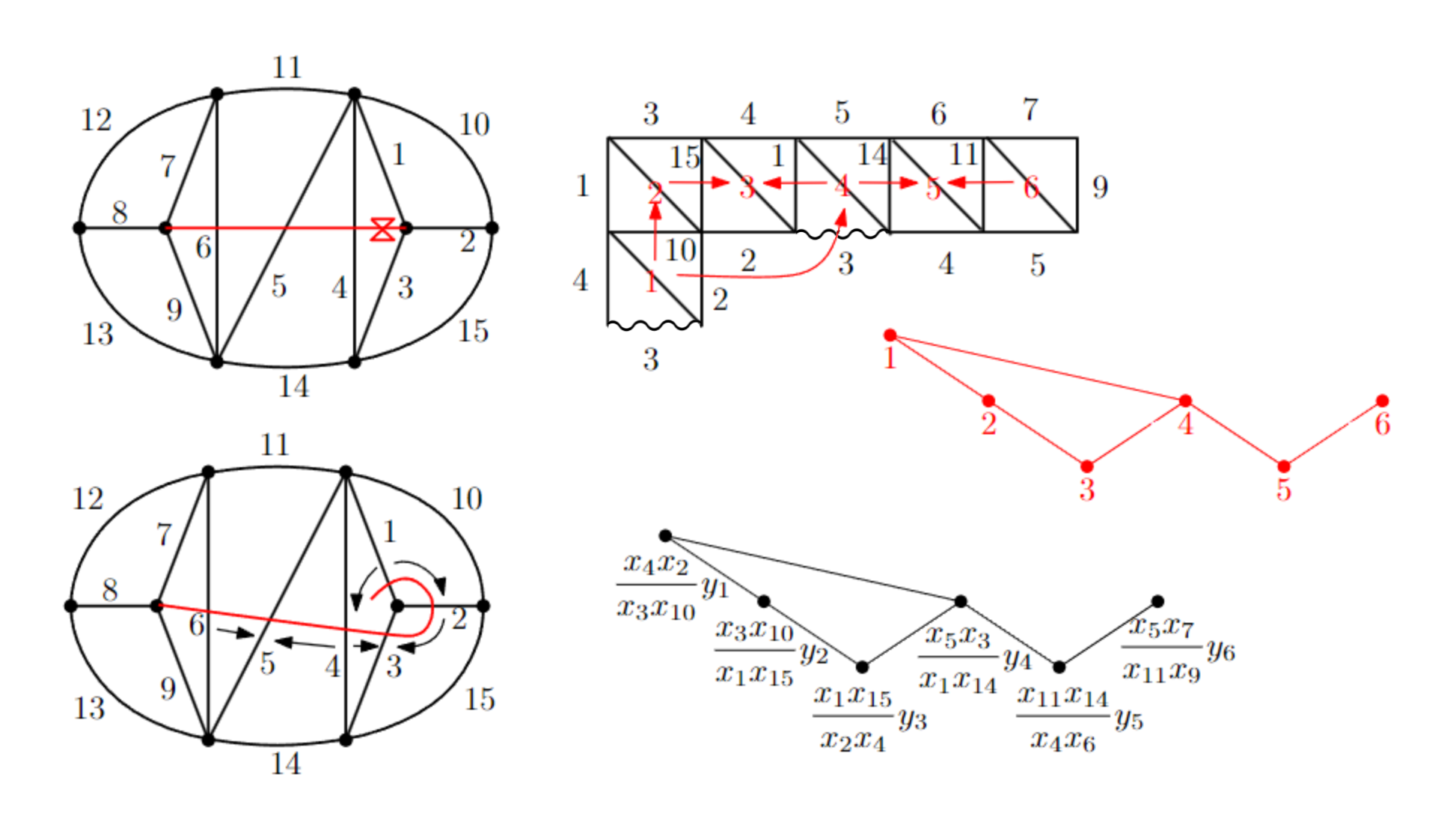}
\caption{An example with a single notched arc, from \cite{wilson}, with corresponding loop graph, poset and labeled poset.} \label{singlenotched}
\end{figure}

 Even though Figure~\ref{singlenotched} shows corresponding snake graph computation, we again note that it is just for convenience as the loop graph can be read directly from the surface. Let us build the underlying poset step by step to get the weight polynomial.

\vspace{5 mm}

\noindent{\textbf{Step $1$: }} We connect the first four vertices with appropriate arrows. We use the upwards arrow for vertex $4$ because we will connect it to vertex $1$ by adding the relation $1\succeq 4$ in Step 2.

\begin{figure}[H]
    \centering
    \scalebox{.8}{
\begin{tikzpicture}
\draw[->, blue, dashed, thick] (0,0)--(1,-2/3);
\draw[->, blue, dashed, thick] (1.5,-1)--(2.5,-5/3);
\draw[->, blue, dashed, thick] (3,-2)--(4,-4/3);
\draw[->, blue, dashed, thick] (4.5,-1)--(5.5,-1/3);
\fill (0,0) circle(.2) node[white] {$1$} ;
\draw (.1,-.5) node{$\displaystyle w_1$} ;%
\fill (1.5,-1) circle(.2) node[white] {$2$} ;
\draw (1.6,-1.5) node{$\displaystyle w_2$} ;
\fill (3,-2) circle(.2) node[white] {$3$} ;
\draw (3.1,-2.5) node{$\displaystyle w_3$} ;
\fill (4.5,-1) circle(.2) node[white] {$4$} ;
\draw (4.6,-1.5) node{$\displaystyle w_4$} ;
\fill[white] (0,0) circle(.2) ;
\fill[red] (0,0) circle(.1) ;
\draw[red] (0,0) circle(.2);
\fill[white] (1.5,-1) circle(.2) ;
\fill[red] (1.5,-1) circle(.1) ;
\draw[red] (1.5,-1) circle(.2);
\fill[white] (3,-2) circle(.2) ;
\fill[red] (3,-2) circle(.1) ;
\draw[red] (3,-2) circle(.2);
\fill[white] (4.5,-1) circle(.2) ;
\fill[red] (4.5,-1) circle(.1) ;
\draw[red] (4.5,-1) circle(.2);
\end{tikzpicture} \raisebox{15mm}{$\qquad$ \scalebox{2}{$\Rightarrow$} $\qquad$}
\begin{tikzpicture}
\draw (0,0)-- (3,-2)--(4.5,-1);
\draw[->, blue, dashed,thick] (4.5,-1)--(5.5,-1/3);
\fill (0,0) circle(.1) ;
\draw (.1,-.5) node{$\displaystyle w_1$} ;%
\fill (1.5,-1) circle(.1) ;
\draw (1.6,-1.5) node{$\displaystyle w_2$} ;
\fill (3,-2) circle(.1) ;
\draw (3.1,-2.5) node{$\displaystyle w_3$} ;
\fill[blue] (4.5,-1) circle(.15);
\draw (4.6,-1.5) node{$\displaystyle w_4$} ;
\fill[white] (0,0) circle(.2) ;
\fill[red] (0,0) circle(.1) ;
\draw[red] (0,0) circle(.2);
\end{tikzpicture}}
\end{figure}


The corresponding matrix is given by $\mdo(w_1)\mdo(w_2)\mup(w_3)\mup(w_4)$.
 \begin{align*}
S_{1} :=  \mdo(w_1)\mdo(w_2)\mup(w_3)\mup(w_4)= \mdom{w_1}\mdom{w_2}\mupm{w_3}\mupm{w_4}.
\end{align*}

\vspace{9 mm}

\noindent{\textbf{Step $2$: }} Now we need to connect the upwards pointing arrow from vertex $4$ with vertex $1$ to form a loop. We will use the source loop operation as we want to keep making connections using vertex $4$.
\begin{figure}[H]
    \centering \scalebox{.8}{
\begin{tikzpicture}
\draw (0,0)-- (3,-2)--(4.5,-1);
\draw[->, blue, dashed, thick] (4.5,-1)--(5.5,-1/3);
\fill (0,0) circle(.1) ;
\draw (.1,-.5) node{$\displaystyle w_1$} ;%
\fill (1.5,-1) circle(.1) ;
\draw (1.6,-1.5) node{$\displaystyle w_2$} ;
\fill (3,-2) circle(.1) ;
\draw (3.1,-2.5) node{$\displaystyle w_3$} ;
\fill[blue] (4.5,-1) circle(.15);
\draw (4.6,-1.5) node{$\displaystyle w_4$} ;
\fill[white] (0,0) circle(.2) ;
\fill[red] (0,0) circle(.1) ;
\draw[red] (0,0) circle(.2);
\end{tikzpicture} \raisebox{15mm}{$\qquad$ \scalebox{2}{$\Rightarrow$} $\qquad$}
\begin{tikzpicture}
\draw (0,0)-- (3,-2)--(4.5,-1)--(0,0);
\draw[->, blue, dashed, thick] (4.5,-1)--(5.5,-5/3);
\fill (0,0) circle(.1) ;
\draw (.1,-.5) node{$\displaystyle w_1$} ;%
\fill (1.5,-1) circle(.1) ;
\draw (1.6,-1.5) node{$\displaystyle w_2$} ;
\fill (3,-2) circle(.1) ;
\draw (3.1,-2.5) node{$\displaystyle w_3$} ;
\fill[blue] (4.5,-1) circle(.15);
\draw (4.6,-1.5) node{$\displaystyle w_4$} ;
\fill[white] (4.5,-1) circle(.2) ;
\fill[red] (4.5,-1) circle(.1) ;
\draw[red] (4.5,-1) circle(.2);
\end{tikzpicture}}
\end{figure}

We get the matrix: $S_2=\lloop_\nearrow(S_1)$. The formula for $S_2$ is:
$$  \lloop_\nearrow \left( \mdo(w_1)\mdo(w_2)\mup(w_3)\mup(w_4) \right).$$

\vspace{9 mm}

\noindent{\textbf{Step $3$: }} We connect an up matrix for $5$ and a down matrix for $6$.  An up matrix for $6$ would also work as we are making no more connections, but then we would consider sum of the top row of the resulting matrix as a final answer.

\begin{figure}[H]
    \centering \scalebox{.8}{
\begin{tikzpicture}
\draw (0*.75,0)-- (3*.75,-2)--(4.5*.75,-1)--(0*.75,0);
\draw[->, blue, dashed, thick] (4.5*.75,-1)--(5.5*.75,-5/3);
\fill (0*.75,0) circle(.1) ;
\draw (.1*.75,-.5) node{$\displaystyle w_1$} ;%
\fill (1.5*.75,-1) circle(.1) ;
\draw (1.6*.75,-1.5) node{$\displaystyle w_2$} ;
\fill (3*.75,-2) circle(.1) ;
\draw (3.1*.75,-2.5) node{$\displaystyle w_3$} ;
\draw (4.6*.75,-1.5) node{$\displaystyle w_4$} ;
\fill[white] (4.5*.75,-1) circle(.2) ;
\fill[red] (4.5*.75,-1) circle(.1) ;
\draw[red] (4.5*.75,-1) circle(.2);
\draw (6.1*.75,-2.5) node{$\displaystyle w_5$} ;
\draw (7.6*.75,-1.5) node{$\displaystyle w_6$} ;
\draw[->, blue, dashed, thick] (6*.75,-2)--(7*.75,-4/3);
\draw[->, blue, dashed, thick] (7.5*.75,-1)--(8.5*.75,-5/3);
\fill[white] (6*.75,-2) circle(.2) ;
\fill[red] (6*.75,-2) circle(.1) ;
\draw[red] (6*.75,-2) circle(.2);
\fill[white] (7.5*.75,-1) circle(.2) ;
\fill[red] (7.5*.75,-1) circle(.1) ;
\draw[red] (7.5*.75,-1) circle(.2);
\end{tikzpicture}
 \raisebox{15mm}{$\qquad$ \scalebox{2}{$\Rightarrow$} $\qquad$}
\begin{tikzpicture}
\draw (0*.75,0)-- (3*.75,-2)--(4.5*.75,-1)--(0*.75,0);
\draw(4.5*.75,-1)--(6*.75,-2)--(7.5*.75,-1);
\fill (0*.75,0) circle(.1) ;
\draw (.1*.75,-.5) node{$\displaystyle w_1$} ;%
\fill (1.5*.75,-1) circle(.1) ;
\draw (1.6*.75,-1.5) node{$\displaystyle w_2$} ;
\fill (3*.75,-2) circle(.1) ;
\draw (3.1*.75,-2.5) node{$\displaystyle w_3$} ;
\draw (4.6*.75,-1.5) node{$\displaystyle w_4$} ;
\fill[blue] (7.5*.75,-1) circle(.15) ;
\draw (6.1*.75,-2.5) node{$\displaystyle w_5$} ;
\draw (7.6*.75,-1.5) node{$\displaystyle w_6$} ;
\draw[->, blue, dashed, thick] (7.5*.75,-1)--(8.5*.75,-5/3);
\fill(6*.75,-2) circle(.1) ;
\fill[white] (4.5*.75,-1) circle(.2) ;
\fill[red] (4.5*.75,-1) circle(.1) ;
\draw[red] (4.5*.75,-1) circle(.2);
\end{tikzpicture}}
\end{figure}

We have $S_3=S_2 \cdot \tf \cdot \mup(w_5) \cdot \mdo(w_6)$. We get the following matrix expansion: \begin{align*}
&\lloop_\nearrow \left( \mdom{w_1}\mdom{w_2}\mupm{w_3}\mupm{w_4} \right) \mupm{w_5}\mdom{w_6}.
\end{align*}

The weight polynomial is given by the top left entry of the resulting matrix:
\begin{dmath*}
1+w_3+w_5+w_2w_3+w_3w_5+w_5w_6+w_2w_3w_5+w_3w_4w_5+w_3w_5w_6+w_2w_3w_4w_5+w_2w_3w_5w_6+w_3w_4w_5w_6+w_1w_2w_3w_4w_5+w_2w_3w_4w_5w_6+w_1w_2w_3w_4w_5w_6.
\end{dmath*}
 
 Now we can plug in the weights and obtain the expansion formula:
 \begin{dmath*}
 x_{\gamma}=\displaystyle {\frac{x(M_-)}{\operatorname{cross}(\gamma,T)}} \rank(P_\gamma;xy)=\frac{x_1x_2x_4^2x_6x_9}{x_1x_2x_3x_4x_5x_6}\rank(P_\gamma;xy)=\frac{x_4x_9}{x_3x_5}+\frac{x_1x_9x_{15}}{x_2x_3x_5}y_3+\frac{x_9x_{11}x_{14}}{x_3x_5x_6}y_5+\frac{x_9x_{10}}{x_2x_5}y_2y_3+\frac{x_1x_9x_{11}x_{14}x_{15}}{x_2x_3x_4x_5x_6}y_3y_5+\frac{x_{7}x_{14}}{x_3x_6}y_5y_6+\frac{x_9x_{10}x_{11}x_{14}}{x_2x_4x_5x_6}y_2y_3y_5+\frac{x_9x_{11}x_{15}}{x_2x_4x_6}y_3y_4y_5+\frac{x_1x_{7}x_{14}x_{15}}{x_2x_3x_4x_6}y_3y_5y_6+\frac{x_3x_9x_{10}x_{11}}{x_1x_2x_4x_6}y_2y_3y_4y_5+\frac{x_7x_{10}x_{14}}{x_2x_4x_6}y_2y_3y_5y_6+\frac{x_5x_{7}x_{15}}{x_2x_4x_6}y_3y_4y_5y_6+\frac{x_9x_{11}}{x_1x_6}y_1y_2y_3y_4y_5+\frac{x_3x_5x_{7}x_{10}}{x_1x_2x_4x_6}y_2y_3y_4y_5y_6+\frac{x_5x_7}{x_1x_6}y_1y_2y_3y_4y_5y_6.
 \end{dmath*}
 
 \end{example}
 
\begin{example}\label{ex:double1}

The following example is a continuation of the previous one but now we have a doubly notched arc attached to the punctures. It is almost the same as Example 5.10 from \cite{wilson} up to a relabeling of the vertices, which is done to unify all examples considered in this paper. 

\begin{figure}[ht]
\includegraphics[width=14cm]{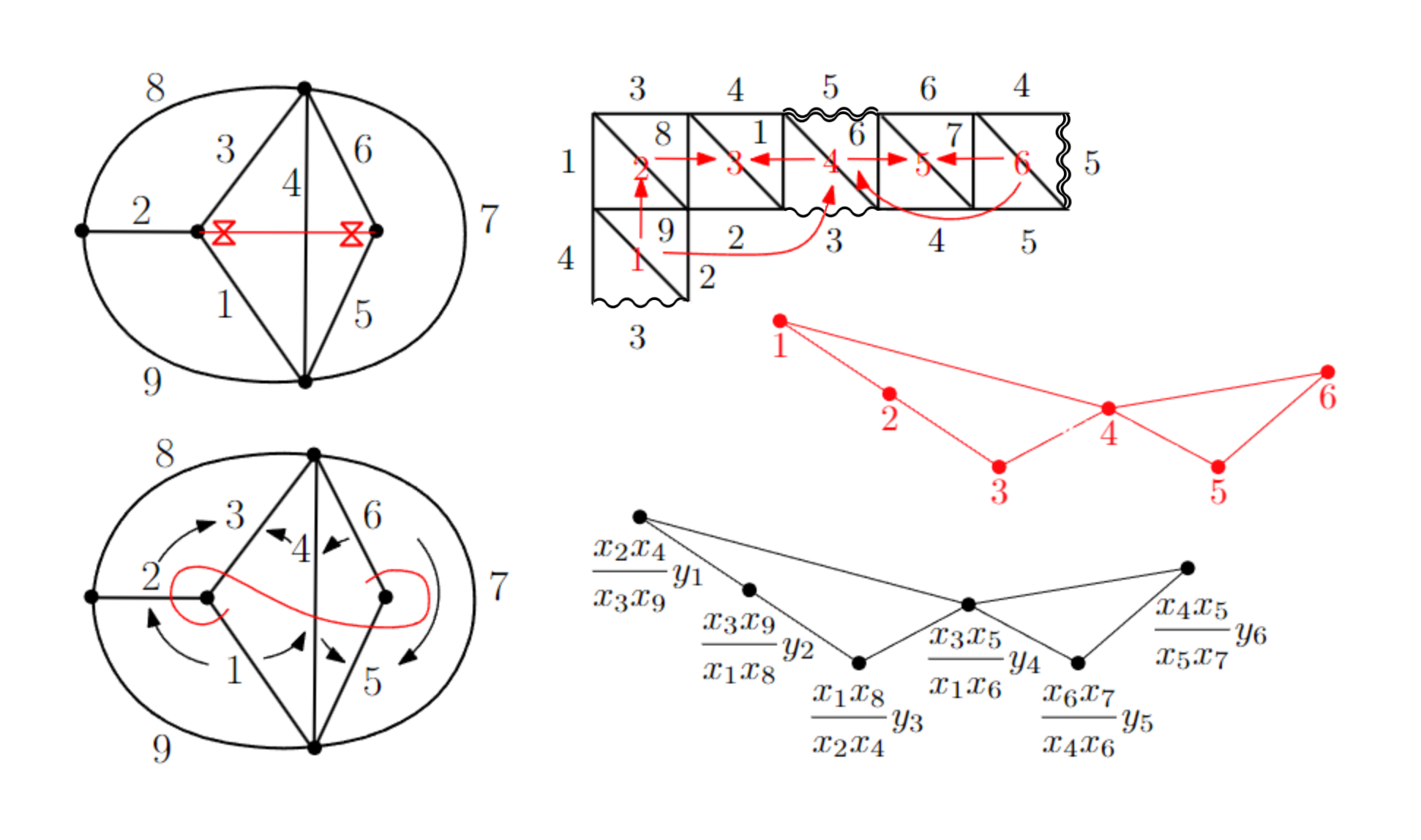}
\caption{A doubly notched arc with its associated loop graph, poset and labeled poset.} \label{doublynotched}
\end{figure}

This is almost the same graph as the above example with one loop, with the only addition being an extra arrow from vertex $6$ to vertex $4$. We can follow the same process up to Step $3$ where we obtained the matrix $S_3$ with the following expansion.

 \begin{align*}
&S_3=\lloop_\nearrow \left( \mdom{w_1}\mdom{w_2}\mupm{w_3}\mupm{w_4} \right)\tfm \mupm{w_5}\mdom{w_6}.
\end{align*}

\noindent{\textbf{Step $4$: }} Take the trace to connect the source vertex $6$ with the target vertex $4$. As the arrow is already pointing downward, there is no need for a tail flip.

\begin{figure}[H]
    \centering \scalebox{.8}{
\begin{tikzpicture}
\draw (0*.75,0)-- (3*.75,-2)--(4.5*.75,-1)--(0*.75,0);
\draw(4.5*.75,-1)--(6*.75,-2)--(7.5*.75,-1);
\fill (0*.75,0) circle(.1) ;
\draw (.1*.75,-.5) node{$\displaystyle w_1$} ;%
\fill (1.5*.75,-1) circle(.1) ;
\draw (1.6*.75,-1.5) node{$\displaystyle w_2$} ;
\fill (3*.75,-2) circle(.1) ;
\draw (3.1*.75,-2.5) node{$\displaystyle w_3$} ;
\draw (4.6*.75,-1.5) node{$\displaystyle w_4$} ;
\fill[blue] (7.5*.75,-1) circle(.15) ;
\draw (6.1*.75,-2.5) node{$\displaystyle w_5$} ;
\draw (7.6*.75,-1.5) node{$\displaystyle w_6$} ;
\draw[->, blue, dashed, thick] (7.5*.75,-1)--(8.5*.75,-5/3);
\fill(6*.75,-2) circle(.1) ;
\fill[white] (4.5*.75,-1) circle(.2) ;
\fill[red] (4.5*.75,-1) circle(.1) ;
\draw[red] (4.5*.75,-1) circle(.2);
\end{tikzpicture}
 \raisebox{15mm}{$\qquad$ \scalebox{2}{$\Rightarrow$} $\qquad$}
\begin{tikzpicture}
\draw (0*.75,0)-- (3*.75,-2)--(4.5*.75,-1.1)--(0*.75,0);
\draw(4.5*.75,-1.1)--(6*.75,-2)--(7.5*.75,-.7);
\fill (0*.75,0) circle(.1) ;
\draw (.1*.75,-.5) node{$\displaystyle w_1$} ;%
\fill (1.5*.75,-1) circle(.1) ;
\draw (1.6*.75,-1.5) node{$\displaystyle w_2$} ;
\fill (3*.75,-2) circle(.1) ;
\draw (3.1*.75,-2.5) node{$\displaystyle w_3$} ;
\draw (4.6*.75,-1.6) node{$\displaystyle w_4$} ;
\fill (4.5*.75,-1.1) circle(.1) ;
\fill (7.5*.75,-.7) circle(.1) ;
\draw (6.1*.75,-2.5) node{$\displaystyle w_5$} ;
\draw (7.6*.75,-1.2) node{$\displaystyle w_6$} ;
\draw (7.5*.75,-.7)--(4.5*.75,-1.1);
\fill(6*.75,-2) circle(.1) ;
\end{tikzpicture}}
\end{figure}

 The trace of $S_3$ gives us the weight polynomial:
 \begin{dmath*}1+w_3+w_5+w_2w_3+w_3w_5+w_2w_3w_5+w_3w_4w_5+w_2w_3w_4w_5+w_3w_4w_5w_6+w_1w_2w_3w_4w_5+w_2w_3w_4w_5w_6+w_1w_2w_3w_4w_5w_6.\end{dmath*}
 
  Plugging in the weights we obtain the following expansion formula:
 \begin{dmath*}
 x_{\gamma}=\displaystyle {\frac{x(M_-)}{\operatorname{cross}(\gamma,T)}} \rank(P_\gamma;xy)=\frac{x_1x_2x_4^2x_6}{x_1x_2x_3x_4x_5x_6}\rank(P_\gamma;xy)=\frac{x_4}{x_3x_5}+\frac{x_1x_8}{x_2x_3x_5}y_3+\frac{x_7}{x_3x_5}y_5+\frac{x_9}{x_2x_5}y_2y_3+\frac{x_1x_7x_8}{x_2x_3x_4x_5}y_3y_5+\frac{x_7x_9}{x_2x_4x_5}y_2y_3y_5+\frac{x_7x_8}{x_2x_4x_6}y_3y_4y_5+\frac{x_3x_7x_9}{x_1x_2x_4x_6}y_2y_3y_4y_5+\frac{x_8}{x_2x_6}y_3y_4y_5y_6+\frac{x_7}{x_1x_6}y_1y_2y_3y_4y_5+\frac{x_3x_9}{x_1x_2x_6}y_2y_3y_4y_5y_6+\frac{x_4}{x_1x_6}y_1y_2y_3y_4y_5y_6.
 \end{dmath*}

\end{example}
 
 \begin{example} \label{ex:double2} Here we consider another doubly notched arc; this time with two disjoint loops in the corresponding loop graph, illustrated in Figure~\ref{doubly2}. 
 
\begin{figure}[H]
\includegraphics[width=15cm]{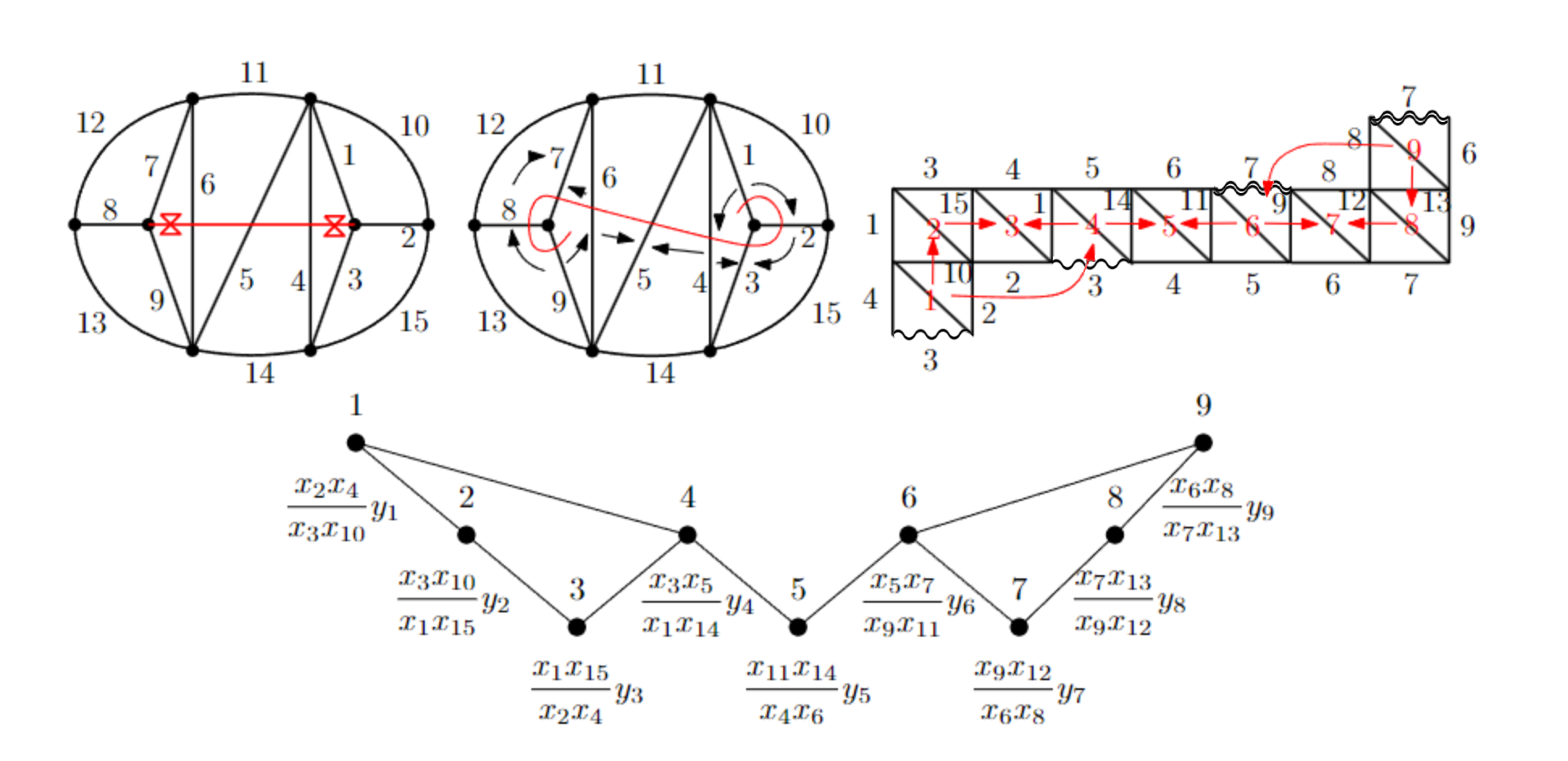}
\caption{An example calculation for a doubly notched arc, from \cite{wilson}, with its corresponding labeled poset.} \label{doubly2}
\end{figure}

The source loop is the one we calculated in Step 2 of the one loop case, with the corresponding matrix: 

$$S_2=\lloop_\nearrow \left( \mdo(w_1)\mdo(w_2)\mup(w_3)\mup(w_4) \right).$$

Let us know similarly calculate the target loop.

\noindent{\textbf{Step $3'$: }} Connect vertices $6$, $7$, $8$ and $9$ by up and down arrows. We use the downwards arrow for vertex $9$ because we will connect it to vertex $6$ by adding the relation $6\preceq 9$ later.

\begin{figure}[H]
    \centering \scalebox{.8}{
\begin{tikzpicture}
\draw[->, blue, dashed, thick] (6,0)--(7,-2/3);
\draw[->, blue, dashed, thick] (1.5,-1)--(2.5,-5/3);
\draw[->, blue, dashed, thick] (3,-2)--(4,-4/3);
\draw[->, blue, dashed, thick] (4.5,-1)--(5.5,-1/3);
\fill (6,0) circle(.2) node[white] {$9$} ;
\draw (6.1,-.5) node{$\displaystyle w_9$} ;%
\fill (1.5,-1) circle(.2) node[white] {$6$} ;
\draw (1.6,-1.5) node{$\displaystyle w_6$} ;
\fill (3,-2) circle(.2) node[white] {$7$} ;
\draw (3.1,-2.5) node{$\displaystyle w_7$} ;
\fill (4.5,-1) circle(.2) node[white] {$8$} ;
\draw (4.6,-1.5) node{$\displaystyle w_8$} ;
\fill[white] (6,0) circle(.2) ;
\fill[red] (6,0) circle(.1) ;
\draw[red] (6,0) circle(.2);
\fill[white] (1.5,-1) circle(.2) ;
\fill[red] (1.5,-1) circle(.1) ;
\draw[red] (1.5,-1) circle(.2);
\fill[white] (3,-2) circle(.2) ;
\fill[red] (3,-2) circle(.1) ;
\draw[red] (3,-2) circle(.2);
\fill[white] (4.5,-1) circle(.2) ;
\fill[red] (4.5,-1) circle(.1) ;
\draw[red] (4.5,-1) circle(.2);
\end{tikzpicture} \raisebox{15mm}{$\qquad$ \scalebox{2}{$\Rightarrow$} $\qquad$}
\begin{tikzpicture}
\draw (1.5,-1)-- (3,-2)--(6,0);
\draw[->, blue, dashed,thick] (6,0)--(7,-2/3);
\fill (4.5,-1) circle(.1) ;
\draw (6.1,-.5) node{$\displaystyle w_9$} ;%
\fill (1.5,-1) circle(.1) ;
\draw (1.6,-1.5) node{$\displaystyle w_6$} ;
\fill (3,-2) circle(.1) ;
\draw (3.1,-2.5) node{$\displaystyle w_7$} ;
\fill[blue] (6,0) circle(.15);
\draw (4.6,-1.5) node{$\displaystyle w_8$} ;
\fill[white] (1.5,-1) circle(.2) ;
\fill[red] (1.5,-1) circle(.1) ;
\draw[red] (1.5,-1) circle(.2);
\end{tikzpicture}}
\end{figure}

The correponding matrix is:

$$S_{3'}= \mdo(w_6)\mup(w_7)\mup(w_8)\mdo(w_9)= \mdom{w_6}\mupm{w_7}\mupm{w_8}\mdom{w_9}.$$

\noindent{\textbf{Step $4'`$: }} Connect vertices $6$ and $9$ using the $\rloop$ operation so that the connections remain on the target (left) side:
\begin{figure}[H]
    \centering \scalebox{.8}{
\begin{tikzpicture}
\draw (1.5,-1)-- (3,-2)--(6,0);
\draw[->, blue, dashed,thick] (6,0)--(7,-2/3);
\fill (4.5,-1) circle(.1) ;
\draw (6.1,-.5) node{$\displaystyle w_9$} ;%
\fill (1.5,-1) circle(.1) ;
\draw (1.4,-1.5) node{$\displaystyle w_6$} ;
\fill (3,-2) circle(.1) ;
\draw (3.1,-2.5) node{$\displaystyle w_7$} ;
\fill[blue] (6,0) circle(.15);
\draw (4.6,-1.5) node{$\displaystyle w_8$} ;
\fill[white] (1.5,-1) circle(.2) ;
\fill[red] (1.5,-1) circle(.1) ;
\draw[red] (1.5,-1) circle(.2);
\end{tikzpicture} \raisebox{15mm}{$\qquad$ \scalebox{2}{$\Rightarrow$} $\qquad$}
\begin{tikzpicture}
\draw (1.5,-1)-- (3,-2)--(6,0)--(1.5,-1);
\draw[->, blue, dashed,thick] (1.5,-1)--(2.1,-2);
\fill (4.5,-1) circle(.1) ;
\draw (6.1,-.5) node{$\displaystyle w_9$} ;%
\fill (1.5,-1) circle(.1) ;
\draw (1.6,-1.5) node{$\displaystyle w_6$} ;
\fill (3,-2) circle(.1) ;
\draw (3.1,-2.5) node{$\displaystyle w_7$} ;
\fill (6,0) circle(.1);
\draw (4.6,-1.5) node{$\displaystyle w_8$} ;
\fill[white] (1.5,-1) circle(.2) ;
\fill[red] (1.5,-1) circle(.1) ;
\draw[red] (1.5,-1) circle(.2);
\end{tikzpicture}}
\end{figure}

We get the following matrix:
$$S_{4'}=\rloop\left( \mdo(w_6)\mup(w_7)\mup(w_8)\mdo(w_9) \right)=\rloop\left( \mdom{w_6}\mupm{w_7}\mupm{w_8}\mdom{w_9} \right).$$

\noindent{\textbf{Step $5'$: }} Finally, connect the pieces together adding an up arrow for vertex $5$ in the middle to get the full poset.

\begin{figure}[H]
    \centering \scalebox{.8}{
\begin{tikzpicture}
\draw (0*.75,0)-- (3*.75,-2)--(4.5*.75,-1)--(0*.75,0);
\draw[->, blue, dashed, thick] (4.5*.75,-1)--(5.5*.75,-5/3);
\fill (0*.75,0) circle(.1) ;
\draw (.1*.75,-.5) node{$\displaystyle w_1$} ;%
\fill (1.5*.75,-1) circle(.1) ;
\draw (1.6*.75,-1.5) node{$\displaystyle w_2$} ;
\fill (3*.75,-2) circle(.1) ;
\draw (3.1*.75,-2.5) node{$\displaystyle w_3$} ;
\draw (4.6*.75,-1.5) node{$\displaystyle w_4$} ;
\fill[white] (4.5*.75,-1) circle(.2) ;
\fill[red] (4.5*.75,-1) circle(.1) ;
\draw[red] (4.5*.75,-1) circle(.2);
\draw (6.1*.75,-2.5) node{$\displaystyle w_5$} ;
\draw[->, blue, dashed, thick] (6*.75,-2)--(7*.75,-4/3);
\fill[white] (6*.75,-2) circle(.2) ;
\fill[red] (6*.75,-2) circle(.1) ;
\draw[red] (6*.75,-2) circle(.2);
\draw (7.5*.75,-1)-- (9*.75,-2)--(12*.75,0)--(7.5*.75,-1);
\draw[->, blue, dashed,thick] (7.5*.75,-1)--(8.1*.75,-2);
\fill (10.5*.75,-1) circle(.1) ;
\draw (12.1*.75,-.5) node{$\displaystyle w_9$} ;%
\fill (7.5*.75,-1) circle(.1) ;
\draw (7.6*.75,-.7) node{$\displaystyle w_6$} ;
\fill (9*.75,-2) circle(.1) ;
\draw (9.1*.75,-2.5) node{$\displaystyle w_7$} ;
\fill (12*.75,0) circle(.1);
\draw (10.6*.75,-1.5) node{$\displaystyle w_8$} ;
\fill[white] (7.5*.75,-1) circle(.2) ;
\fill[red] (7.5*.75,-1) circle(.1) ;
\draw[red] (7.5*.75,-1) circle(.2);
\end{tikzpicture}
 \raisebox{15mm}{$\,$ \scalebox{2}{$\Rightarrow$} $\,$}
\begin{tikzpicture}
\draw (0*.7,0)-- (3*.7,-2)--(4.5*.7,-1)--(0*.7,0);
\fill (0*.7,0) circle(.1) ;
\draw (.1*.7,-.5) node{$\displaystyle w_1$} ;%
\fill (1.5*.7,-1) circle(.1) ;
\draw (1.6*.7,-1.5) node{$\displaystyle w_2$} ;
\fill (3*.7,-2) circle(.1) ;
\draw (3.1*.7,-2.5) node{$\displaystyle w_3$} ;
\draw (4.6*.7,-1.5) node{$\displaystyle w_4$} ;
\draw (6.1*.7,-2.5) node{$\displaystyle w_5$} ;
\draw  (4.5*.7,-1)--(6*.7,-2) --(7.5*.7,-1);
\fill (6*.7,-2) circle(.1) ;
\draw (7.5*.7,-1)-- (9*.7,-2)--(12*.7,0)--(7.5*.7,-1);
\draw[->, blue, dashed,thick] (7.5*.7,-1)--(8.1*.7,-2);
\fill (10.5*.7,-1) circle(.1) ;
\draw (12.1*.7,-.5) node{$\displaystyle w_9$} ;%
\fill (7.5*.7,-1) circle(.1) ;
\draw (7.6*.7,-.7) node{$\displaystyle w_6$} ;
\fill (9*.7,-2) circle(.1) ;
\draw (9.1*.7,-2.5) node{$\displaystyle w_7$} ;
\fill (12*.7,0) circle(.1);
\draw (10.6*.7,-1.5) node{$\displaystyle w_8$} ;
\fill[blue] (7.5*.7,-1) circle(.15) ;
\fill[white] (4.5*.7,-1) circle(.2) ;
\fill[red] (4.5*.7,-1) circle(.1) ;
\draw[red] (4.5*.7,-1) circle(.2);
\end{tikzpicture}}
\end{figure}

\begin{align*}
S_{5'}=S_2\cdot \mup(w_5) \cdot S_{4'} = \lloop_\nearrow \left( \mdo(w_1)\mdo(w_2)\mup(w_3)\mup(w_4) \right)\cdot \mup(w_5) \cdot \rloop\left( \mdo(w_6)\mup(w_7)\mup(w_8)\mdo(w_9) \right).
\end{align*}

The upper left entry of the matrix gives us the following rank polynomial

\begin{dmath*}
\rank(P_\gamma;w)=
1 + w_3 + w_5 + w_7 + w_2w_3 + w_3w_5 + w_3w_7 + w_5w_7 + w_7w_8 + w_2w_3w_5 + w_3w_4w_5 + w_2w_3w_7 + w_3w_5w_7 + w_5w_6w_7 + w_3w_7w_8 + w_5w_7w_8 + w_2w_3w_4w_5 + w_2w_3w_5w_7 + w_3w_4w_5w_7 + w_3w_5w_6w_7 + w_2w_3w_7w_8 + w_3w_5w_7w_8 + w_5w_6w_7w_8 + w_1w_2w_3w_4w_5 + w_2w_3w_4w_5w_7 + w_2w_3w_5w_6w_7 + w_3w_4w_5w_6w_7 + w_2w_3w_5w_7w_8 + w_3w_4w_5w_7w_8 + w_3w_5w_6w_7w_8 + w_5w_6w_7w_8w_9  + w_1w_2w_3w_4w_5w_7 + w_2w_3w_4w_5w_6w_7 + w_2w_3w_4w_5w_7w_8 + w_2w_3w_5w_6w_7w_8 + w_3w_4w_5w_6w_7w_8 + w_3w_5w_6w_7w_8w_9  + w_1w_2w_3w_4w_5w_6w_7 + w_1w_2w_3w_4w_5w_7w_8 + w_2w_3w_4w_5w_6w_7w_8 + w_2w_3w_5w_6w_7w_8w_9 + w_3w_4w_5w_6w_7w_8w_9+ w_1w_2w_3w_4w_5w_6w_7w_8 + w_2w_3w_4w_5w_6w_7w_8w_9+w_1w_2w_3w_4w_5w_6w_7w_8w_9.
\end{dmath*}

 \end{example}



\bibliographystyle{alpha}
\bibliography{cluster}

\end{document}